\documentclass[13pt]{article}  % list options between brackets
\usepackage{amssymb}              % list packages between braces
\usepackage{amsthm}
\usepackage{amsmath}
\usepackage{eucal}
\usepackage{verbatim}
\usepackage[dvips]{graphicx}
\usepackage{multirow}
\usepackage{fancyhdr}
\usepackage{color}
\usepackage{enumerate}
\usepackage{calrsfs}
\usepackage{fullpage}
\usepackage{amsmath}
\usepackage{eufrak}

\newtheorem{theorem}{Theorem}
\newtheorem{lemma}{Lemma}
\newtheorem{definition}{Definition}
\newtheorem{remark}{Remark}
\newtheorem{proposition}{Proposition}
\newtheorem{corollary}{Corollary}

\makeatletter
\newcommand{\leqnomode}{\tagsleft@true}
\newcommand{\reqnomode}{\tagsleft@false}
\makeatother
\def\({\begin{eqnarray}}
\def\){\end{eqnarray}}
\def\[{\begin{eqnarray*}}
\def\]{\end{eqnarray*}}
\def\part#1#2{\frac{\partial #1}{\partial #2}}

\def\R{\mathbb{R}}
\def\N{\mathbb{N}}
\def\d{\mathrm{d}}
\def\tot#1#2{\frac{\d #1}{\d #2}}
\def\eps{\varepsilon}

\def\Norm#1{\left\| #1 \right\|}

\def\wtx{\widetilde x}
\def\wx#1#2{{\wtx_{#1}}^{\; #2}}
\def\wxji{\wx{j}{i}}

\def\wtv{\widetilde v}
\def\wv#1#2{{\wtv_{#1}}^{\; #2}}
\def\wvji{\wv{j}{i}}

\def\ta#1#2{\tau_{{#1}{#2}}}
\def\taij{\ta{i}{j}}

\def\c{{\mathfrak{c}}}
\def\cst{{\c^\ast}}
\def\s{{\mathfrak{s}}}
\def\m{{\mathfrak{m}}}

\def\grad{\nabla}

\def\solx{\mathbf{x}}
\def\solv{\mathbf{v}}

\def\dx{{d_{\solx}}}
\def\dv{{d_{\solv}}}

\def\upsi{\underline{\psi}}
\def\otau{\overline{\tau}}

\def\bb{\eta}
\def\cc{\kappa}
\def\dd{\sigma}
\def\taust{\tau^\ast}
\def\psist{\psi^\ast}

\def\POsT{{\mathcal{P}(\Omega_\s^T})}
\def\POsTI{{\mathcal{P}_\s^T[\rho^0]}}
\def\POsZ{{\mathcal{P}(\Omega_\s^0})}
\def\ZZ{\mathcal{Z}}

\def\cJH#1{\textcolor{blue}{\bf [#1]}}

\begin{document}

\title{Cucker-Smale model with finite speed of information propagation: well-posedness, flocking and mean-field limit}   % type title between braces
\author{Jan Haskovec\footnote{Computer, Electrical and Mathematical Sciences \& Engineering, King Abdullah University of Science and Technology, 23955 Thuwal, KSA.
jan.haskovec@kaust.edu.sa}}

\maketitle

\begin{abstract}
We study a variant of the Cucker-Smale model where information between agents
propagates with a finite speed $\c>0$. This leads to a system of functional differential equations
with state-dependent delay.
%We prove three results:
We prove that, if initially the agents travel slower than $\c$,
then the discrete model admits unique global solutions.
Moreover, under a generic assumption on the influence function,
we show that there exists a critical information propagation speed $\cst>0$
such that if $\c\geq\cst$, the system exhibits asymptotic flocking
in the sense of the classical definition of Cucker and Smale.
For constant initial datum the value of $\cst$ is explicitly calculable.
Finally, we derive a mean-field limit of the discrete system,
which is formulated in terms of probability measures
on the space of time-dependent trajectories.
We show global well-posedness of the mean-field problem
and argue that it does not admit a description
in terms of the classical Fokker-Planck equation.
\end{abstract}
\vspace{2mm}

\vspace{2mm}

%SIAM Journal on Control and Optimization
%SIAM Journal on Applied Dynamical Systems
%SIAM J. Applied Dyn Syst, J. Differential Equations

%\textbf{2010 MR Subject Classification}: %34K05, 82C22, 34D05, 92D50.
\vspace{2mm}

%%%%%%%%%%%%%%%%%%%%%%%%%%%%%%%%%%%%%%%%
\section{Introduction}\label{sec:Intro}
The Cucker-Smale system is a prototypical individual-based model of collective behavior.
It was introduced in the seminal papers~\cite{CS1, CS2}, originally as a model for language evolution. 
Later the interpretation as a model for flocking in animals (birds) prevailed.
The model describes a group of $N\in\N$ autonomous agents
located in the physical space $\R^d$, $d\geq 1$.
The agents are described by their phase-space 
coordinates $(x_i(t), v_i(t))\in\R^{2d}$, $i=1,2,\dots,N$,
where $x_i(t)$ denotes the position and $v_i(t)$ the velocity of the $i$-th agent.
The agents are subject to the following collective dynamics,
\(
   \tot{x_i(t)}{t} &=& v_i(t), \label{eq:CS001} \\
   \tot{v_i(t)}{t} &=& \frac{1}{N-1} \sum_{j=1}^N \psi(|x_j(t)-x_i(t)|) (v_j(t)-v_i(t)),  \label{eq:CS002}
\)
for $t>0$ and $i\in[N]$, where here and in the sequel we denote $[N] := \{1, \ldots, N\}$.
The nonnegative real function $\psi:[0,\infty)\to [0,\infty)$,
called \emph{influence function}, measures the intensity of the influence between
agents depending on their distance.
%Typically, it is assumed that the influence function decreases with increasing distance.
Typically, it is assumed to be globally bounded, so that, by an eventual rescaling of time,
one has $0\leq\psi\leq 1$. We shall adapt this assumption in our paper.

The terms $x_j(t)$ and $v_j(t)$ appearing in the right-hand side of equation \eqref{eq:CS002}
reflect the implicit modelling assumption that each agent receives information
about the current phase-space co-ordinates of all other agents immediately, without any time lags
(or, that the eventual time lags are negligibly small with respect to the typical time scales
relevant for the system).
However, for certain applications in biology \cite{Camazine, Smith, Vicsek}, socio-economics \cite{Krugman, Naldi}
or engineering (e.g., swarm robotics, \cite{Hamman, Jadbabaie, E3B, Valentini}), delays in communication between agents
caused by finite speed of information propagation may be relevant.
For instance, in radio communication between satellites on the orbit or in outer space,
where the distances are not negligible with respect to the speed of light,
or in swarm robotics with acoustic communication between agents (i.e., underwater robots).
This motivates us to introduce a modification of the Cucker-Smale model \eqref{eq:CS001}--\eqref{eq:CS002}
where information propagates with a constant finite speed $\c>0$, called \emph{propagation speed} in the sequel.
Then, agent $i$ located at $x_i=x_i(t)$ at time $t>0$ observes the
phase-space co-ordinates of agent $j$ at time $t-\taij(t)$, where $\taij=\taij(t)$ solves
\(  \label{eq:tau}
   \c \taij(t) = |x_i(t) - x_j(t-\taij(t))|.
\)
In other words, $\taij(t)$ is the time that information (light, sound) needs to travel from location $x_j(t-\taij(t))$
to location $x_i(t)$.

In general it is neither guaranteed that a solution $\taij(t)$ of \eqref{eq:tau}
exists nor that it is unique. This issue is of course related
to the possibility of agents traveling faster than the propagation speed $\c$.
We shall formulate sufficient conditions for the well-posedness
of the model in the course of our analysis.
For the time being, let us assume that \eqref{eq:tau} is uniquely solvable with solution $\taij(t)\geq 0$
for all $t\geq 0$ and $i,j\in[N]$, and introduce the following notation
\[
   \wxji := x_j(t - \taij(t)), \qquad \wvji := v_j(t - \taij(t)).
\]
We also introduce the formal notation $\ta{i}{i}:=0$ and $\wx{i}{i}:=x_i(t)$, $\wv{i}{i}:=v_i(t)$,
and, if no danger of confusion, we shall usually drop the explicit time dependence,
writing just $x_i$ for $x_i(t)$ etc.
With this notation, the system that we study in this paper is written as
\(
   \dot x_i &=& v_i,   \label{eq:CS1} \\
   \dot v_i &=& \frac{1}{N-1} \sum_{j=1}^N \psi(|\wxji - x_i|) \left( \wvji - v_i \right),  \label{eq:CS2}
\)
for $t>0$ and $i\in[N]$.
We shall also frequently use the shorthand notation for the \emph{communication rates}
\[  %\label{eq:psiij}
   \widetilde\psi_{ij} := \psi\left( \left| \wxji - x_i \right| \right).
\]
The system \eqref{eq:CS1}--\eqref{eq:CS2} is equipped with the initial datum
\(  \label{IC:0}
    x_i(t) = x_i^0(t), \quad v_i(t) = v_i^0(t) \qquad\mbox{for } i\in[N],\quad t\leq 0,
\)
where $x_i^0=x_i^0(t)$, $v_i^0=v_i^0(t)$ are continuous paths on $(-\infty,0]$.
We shall impose the physically relevant assumption that
\(   \label{IC:comp}
   x^0_i(t) = x^0_i(0) + \int_0^t v^0_i(s) \d s \qquad\mbox{for } i\in[N],\quad t\leq 0,
\)
so that we in fact only prescribe $v_i^0=v_i^0(t)$ for $t\leq 0$ and $x_i^0(0)$.
We note that the results provided in this paper would not lose their validity
even if we dropped the "compatibility" assumption \eqref{IC:comp},
however, it would lead to unnecessary technicalities in the proofs.
Moreover, in order to construct the solution of \eqref{eq:CS1}--\eqref{eq:CS2}
on a bounded time interval $[0,T]$, the values of the initial datum are only relevant
on a bounded time interval $[-S(T),0]$, with some $S(T)>0$ depending on $\c$, $\s$
and the configuration of the system at time $t=0$.
We shall make this dependence explicit later.

From the mathematical point of view, the system \eqref{eq:tau}--\eqref{eq:CS2}
is a system of functional differential equations with state-dependent delay,
which stems from the fact that the delay $\tau_{ij}(t)$ in \eqref{eq:CS2}
depends on the configuration of the system
in a nontrivial (even implicit) way through \eqref{eq:tau}.
This poses new analytical challenges: in particular,
the standard well-posedness theory for ODE systems
(the classical theorems of Peano and Picard-Lindel\"of)
does not apply to \eqref{eq:tau}--\eqref{eq:CS2}.
Therefore, the first goal of this paper, addressed in Section \ref{sec:ex},
is to establish global existence and uniqueness of solutions for \eqref{eq:tau}--\eqref{eq:CS2}.

The second goal of this paper is to study the asymptotic
behavior of solutions of the system \eqref{eq:tau}--\eqref{eq:CS2}.
In particular, the usual question asked in the context
of the Cucker-Smale model is under which conditions
its solutions exhibit the \emph{(asymptotic) flocking behavior},
defined as follows.

\begin{definition}[Asymptotic flocking]\label{def:flocking}
We say that the system with particle positions $x_i(t)$ and velocities $v_i(t)$,
$i\in[N]$ and $t\geq 0$, exhibits \emph{(asymptotic) flocking} if
\(  \label{flocking}
   \sup_{t\geq 0} \dx(t) < \infty,\qquad \lim_{t\to\infty} \dv(t) = 0,
\)
where we denoted the spatial and, resp., velocity diameters
\(  \label{dXdV}
   \dx(t) := \max_{1 \leq i,j \leq N}|x_i(t) - x_j(t)| \quad \mbox{and} \quad \dv(t) := \max_{1 \leq i,j \leq N}|v_i(t) - v_j(t)|.
\)
\end{definition}

We shall derive sufficient conditions for flocking in the system \eqref{eq:tau}--\eqref{eq:CS2}
in terms of the propagation speed $\c$, the decay properties of the influence function $\psi$
and certain properties of the initial datum \eqref{IC:0}. The presence of the state-dependent delay
poses severe challenges for the analysis. In particular, the Lyapunov functional-type approaches,
that were developed for Cucker-Smale-type systems with a-priori given (i.e., state-independent)
delay, fail. Consequently, new methods need to be developed to study the asymptotic flocking
for the system \eqref{eq:tau}--\eqref{eq:CS2}. This goal will be addressed in Section \ref{sec:flocking}.

Finally, the third goal of this paper, addressed in Section \ref{sec:MF},
is to study the mean-field limit of \eqref{eq:tau}--\eqref{eq:CS2}
as $N\to\infty$. As is well known, mean-field limits of systems of interacting particles
are typically described in terms of a Fokker-Planck equation that governs the evolution
of a time-dependent particle density in the phase space.
However, for certain types of systems involving delay, such a description may not be available.
For a Cucker-Smale system with constant delay, this issue was discussed in \cite[Section 4]{HasMar}.
It is not a-priori clear whether the system \eqref{eq:tau}--\eqref{eq:CS2} with state-dependent delay
admits a Fokker-Planck-type description in the mean-field limit.
We approach the problem by first deriving a mean-field limit formulated
in terms of probability measures on the space of Lipschitz continuous curves.
In Section \ref{sec:MF} we prove the well-posedness of such a description
and, using a stability result, show that it is indeed obtained from the discrete system
\eqref{eq:tau}--\eqref{eq:CS2} in the limit as $N\to\infty$.
Many of the analytical techniques applied in Section \ref{sec:MF} are adaptations
of the methods developed in \cite{CCR} to the setting
of probability measures on the space of Lipschitz continuous curves.
Finally, we argue that the mean-field limit in fact \emph{cannot} be formulated
as a (standard) Fokker-Planck equation.

Cucker-Smale-type systems with delay and their flocking behavior have been studied in a series of recent papers.
However, to our best knowledge, all the previous works
\cite{Cartabia, ChoiH1, ChoiH2, Choi-Pignotti, EHS, H:SIADS, HasMar, HasMar2, Liu-Wu, Pignotti-Reche1, Pignotti-Reche2, Pignotti-Trelat, E6}
assume the delay to be state-independent, i.e., given a-priori either as a constant,
time-dependent function or probability distribution.
State-dependent delay induced by finite speed of information propagation
was considered in \cite{Has:sdHK} for a Hegselmann-Krause-type model
of consensus formation \cite{HK}. This model can be seen as a first-order version
of \eqref{eq:tau}--\eqref{eq:CS2}. However, the mathematical challenges
stemming from the second-order model are significantly different.
Consequently, analysis of the system \eqref{eq:tau}--\eqref{eq:CS2}
requires development of new techniques, which is the main focus of this paper.

The paper is organized as follows: In Section \ref{sec:overview}
we present an overview of our main results on well-posedness,
flocking behavior and mean-field limit of the system \eqref{eq:tau}--\eqref{eq:CS2}.
In Section \ref{sec:ex} we provide the proof of well-posedness
(global existence and uniqueness) of solutions of the discrete system.
In Section \ref{sec:flocking} we provide the proof of our result
on asymptotic flocking behavior. % of the solutions.
Finally, in Section \ref{sec:MF} we study the
mean-field limit $N\to\infty$.

%%%%%%%%%%%%%%%%%%%%%%%%%%%%%%%
\section{Overview of main results}\label{sec:overview}

\textbf{Notation.}
In the sequel we shall denote $\solx = (x_1, \ldots, x_N)\in \R^{Nd}$ the vector
of position trajectories and $\solv = (v_1, \ldots, v_N)\in \R^{Nd}$ the vector of velocity trajectories.
By $C(\mathcal{I}; \R^d)$ we denote the space of continuous functions on the interval $\mathcal{I}$ with values in $\R^d$,
and by $C^1_b(\mathcal{I}; \R^d)$ the space of continuous uniformly bounded functions on $\mathcal{I}$ with
continuous first-order derivative.
By $C_\s(\mathcal{I}; \R^d)$ we denote the set of Lipschitz continuous functions on $\mathcal{I}$
with Lipschitz constant $\s\geq 0$; we shall also often use the abbreviated expression
``$\s$-Lipschitz continuous functions on $\mathcal{I}$''.
 The notation $\solv\in C_\s(\mathcal{I}; \R^{Nd})$ is understood
as the space of $N$ trajectories in the $d$-dimensional space, such that $v_i\in C_\s(\mathcal{I}; \R^d)$
for all $i\in[N]$, where $[N]:= \{1,\ldots,N\}$.

\subsection{Existence and uniqueness of solutions}
We start with the observation that, in general, \eqref{eq:tau} can only be uniquely solvable if the agents
move with speeds $|v_i|$ strictly less than the propagation speed $\c$.
This motivates us, for $0 < \s < \c$ and $T\geq 0$, to introduce the space
\(  \label{def:Vs}
   \mathbb{V}_\s^T := \left\{ \solv\in C((-\infty,T];\R^{Nd});\, |v_i(t)| \leq \s \mbox{ for all } t\leq T,\, i\in [N] \right\},
\)
where we use the notation $\solv = (v_1, \ldots, v_N)\in \R^{Nd}$.
Then, it is natural to require that the initial datum $\solv^0 = (v^0_1, \ldots, v^0_N)$ in \eqref{IC:0} is an element of $\mathbb{V}_\s^0$
for some $\s<\c$.
%In fact, we shall see that the initial datum is only relevant on the finite time
%interval $[-S^0, 0]$ where $S^0=\frac{\dx(0)}{\s-\c}$ and the position diameter $\dx$ defined in \eqref{dXdV}.

Our first result provides existence and uniqueness of global solutions of the system \eqref{eq:tau}--\eqref{IC:comp}.

\begin{theorem}[Existence and uniqueness of global solutions]  \label{thm:ex}
Let the influence function $0 \leq \psi \leq 1$ be uniformly Lipschitz continuous on $[0,\infty)$.
Let $0<\s<\c$ be fixed and let the initial datum $\solv^0 = (v^0_1, \ldots, v^0_N) \in \mathbb{V}_\s^0$.
%with $S^0=\frac{\dx(0)}{\s-\c}$.
Then, for any $T>0$ the system \eqref{eq:tau}--\eqref{IC:comp} admits a unique global solution in $\mathbb{V}_\s^T$.

%Moreover, if $\c \geq 3\s$, then the initial datum $\solv^0$ only needs to be prescribed on the interval $\left[-\frac{\dx(0)}{\c-\s}, 0\right]$.
Moreover, the initial datum $\solv^0$ only needs to be prescribed on the compact interval $\left[-S(T), 0\right]$, with
\(  \label{ST}
   S(T) := \frac{\dx(0) + [\c-3\s]^- T}{\c-\s},
\)
with $[a]^- := \max\{-a,0\}$ denotes the negative part of $a$.
\end{theorem}

\subsection{Flocking behavior}
Our second result provides sufficient conditions for asymptotic flocking behavior of the system \eqref{eq:tau}--\eqref{IC:comp} in sense of Definition \ref{def:flocking}.
For this, we adopt the additional assumption that the initial datum $\solv^0$ is constant,
while $\solx^0$ verifies \eqref{IC:comp}. This assumption may be seen restrictive, although we do not consider
it unrealistic. In fact, a significant part of the proof of the flocking result in Section \ref{sec:flocking}
applies to the general setting $\solv^0 \in \mathbb{V}_\s^0$ and does not require $\solv^0$ to be constant.
Only in the last step (Lemma \ref{lem:notsocrazy}), a solution of a nonlinear system of algebraic equations
has to be found, which only seems to be achievable analytically %(and, in fact, explicitly)
for constant $\solv^0$.
For the general case, one would need to resort to numerical approaches to confirm solvability of the algebraic system
and so derive a sufficient condition for flocking in terms of the parameter values; see Remark \ref{rem:v0} for details.

\begin{theorem}[Critical propagation speed for flocking] \label{thm:flocking}
Let the influence function $0 \leq \psi \leq 1$ be uniformly Lipschitz continuous on $[0,\infty)$.
Let $\s>0$ be fixed and let the initial datum
%$\solv^0 = (v^0_1, \ldots, v^0_N) \in \mathbb{V}_\s^0$.
$\solv^0 = (v^0_1, \ldots, v^0_N)$ be constant on $(-\infty,0]$ with $\max_{i\in[N]} |v_i^0| \leq \s$.
Assume that there exists $\bb>0$ such that
\(   \label{ass:bb}
   \psi\left(\dx(0) + \frac{\dv(0)}{\bb}\right) > \bb.
\)
Then there exists $\cst>\s$, calculable from the values of $\dx(0)$, $\dv(0)$, $\s$ and $\bb$, % and the decay properties of $\psi$,
such that if $\c\geq \cst$, the system \eqref{eq:tau}--\eqref{IC:comp}
exhibits flocking in the sense of Definition \ref{def:flocking}.
%Moreover, the initial datum $\solv^0$ only needs to be prescribed on a bounded time interval,
%with length depending on the quantities listed above.
\end{theorem}

%Remark about IC relevant on bounded interval does not make sense - since it is assumed to be constant!!

Let us note that the assumption \eqref{ass:bb} is very natural
in the context of the Cucker-Smale model,
where slow enough decay of the influence function
is necessary (and sufficient) condition for (unconditional) flocking.
In particular, we have the following result.

\begin{proposition}\label{corr:CS}
Let the influence function $\psi=\psi(s)$ be such that
\(  \label{corr:CS:cond}
   \liminf_{s\to\infty} \frac{\psi(s)}{s^\alpha} > 0
\)
for some $\alpha>-1$.
Then for any $\dx(0)$, $\dv(0)\geq 0$ there exists $\bb>0$ such that \eqref{ass:bb} holds.
\end{proposition}

Condition \eqref{corr:CS:cond} is almost equivalent to the unconditional flocking
condition $\int^{\infty} \psi(s) \d s = \infty$ for the classical Cucker-Smale model,
which is known to be sharp \cite{CS1, CS2, Tadmor-Ha}.
Indeed, the nonintegrability of $\psi$ at infinity is implied by \eqref{corr:CS:cond}
with any $\alpha \geq -1$, so that Proposition \ref{corr:CS} only excludes
influence functions that decay like $s^{-1}$.
In particular, for the generic choice of $\psi$ introduced in~\cite{CS1, CS2}
and considered in most of  the subsequent papers,
\(  \label{CS:psi}
   \psi(s) = \frac{1}{(1+s^2)^\beta},
\)
assumption \eqref{corr:CS:cond} of Proposition \ref{corr:CS}, and thus assumption \eqref{ass:bb}
of Theorem \ref{thm:flocking}, is equivalent to $\beta < 1/2$.
This is precisely the condition for {unconditional flocking} formulated in \cite{CS1, CS2}
for the classical Cucker-Smale model.

%Let us also note that the value of the critical propagation speed $\cst$, provided by Theorem \ref{thm:flocking},
%can be computed as a solution of a system of two nonlinear algebraic equations, as detailed in Lemma \ref{lem:crazy} below.
%In particular, one can find the optimal (i.e., minimal) value of $\cst$ by employing numerical optimization techniques.

%%%%%%%%%%%%%%%%%%%%%%%%%%%%%%%%%%%%%
\subsection{Mean-field limit}
For fixed $\s>0$ and $T\geq 0$ we introduce the set
%following subset of the space of $\R^d$-valued $C^1$-functions on $(-\infty,T]$ that are globally Lipschitz continuous and their first-order derivative is Lipschitz continuous on $[0,T]$,
\(   \label{def:OmegasT}
   \Omega_\s^T := %\left\{ \gamma\in C^1((-\infty,T]);\; \gamma=\gamma(t) \mbox{ is $\s$-Lipschitz continuous on $(-\infty, T]$} \right\}.
      \Bigl\{ \gamma\in C_b^1((-\infty, T];\R^d) \cap C_\s((-\infty,T]; \R^d); \;   \dot\gamma|_{[0,T]} \in C_{2\s}([0,T]; \R^d) \Bigr\},
\)
where we recall that $C_b^1((-\infty, T];\R^d)$ denotes the space of continous bounded functions on $(-\infty,T]$
with continuous first-order derivative, and $C_\s((-\infty,T]; \R^d)$ denotes the space of (globally) Lipschitz continuous functions on $(-\infty,T]$
with Lipschitz constant $\s$.
We equip the set $\Omega_\s^T$ with the topology induced by the norm $\Norm{\cdot}_{\Omega_\s^T}$,
\(   \label{def:norm}
   \Norm{\gamma}_{\Omega_\s^T}:= \Norm{\gamma}_{L^\infty(-\infty,T)} + \Norm{\dot \gamma}_{L^\infty(0,T)} \qquad
   \mbox{for } \gamma\in \Omega_\s^T.
\)
We note that with the particular choice $T:=0$ in \eqref{def:OmegasT} we obtain the space $\Omega_\s^0 = C_b^1((-\infty, 0];\R^d) \cap C_\s((-\infty,0]; \R^d)$,
and \eqref{def:norm} reduces to $\Norm{\gamma}_{\Omega_\s^0}= \Norm{\gamma}_{L^\infty(-\infty,0)}$. 

We denote $\POsT$ the space of probability measures on $\Omega_\s^T$ with finite first-order moment, i.e.,
\[
   \int_{\Omega_\s^T} \Norm{\gamma}_{\Omega_\s^T} \d\rho(\gamma) < +\infty \qquad\mbox{for any } \rho\in\POsT.
\]
In Section \ref{sec:MF} we will show that the mean-field limit as $N\to\infty$ of the discrete system \eqref{eq:tau}--\eqref{IC:comp} is represented
by the measure $\rho\in\POsT$, 
\(   \label{eq:law}
   \rho=\mbox{law}(x),
\)
with
\(  \label{eq:MF}
   \dot x(t) = v(t), \qquad  \dot v(t) = F_t[\rho](x(t),v(t)) 
\)
for $t\in (0,T)$,
with the operator $F_t[\rho]: \R^d\times\R^d \to \R^d$,
\(  \label{def:F}
   F_t[\rho](x,v) := \int_{\Omega^T_\s} \psi\left( \left|\Gamma_{t,x}[\gamma] - x \right|\right) \left(\Pi_{t,x}[\gamma] - v \right)  \d\rho(\gamma).
\)
The mapping $\Gamma_{t,x}: \Omega^T_\s \mapsto \R^d$ is defined, for $t\in\R$ and $x\in\R^d$, as
\(  \label{def:Gamma}
   \Gamma_{t,x}[\gamma] := \gamma(t-\tau_{t,x}[\gamma])
\)
with $\tau_{t,x}[\gamma]:=\tau$ the unique solution of
\(   \label{eq:tau_gamma}
   \c\tau = |x - \gamma(t-\tau)|.
%   \c\tau = |x - \gamma(t-\tau)|.
\)
As we shall prove below, the existence and uniqueness of $\tau$ is guaranteed by the $\s$-Lipschitz continuity of $\gamma\in\Omega^T_\s$.
The mapping $\Pi_{t,x}: \Omega^T_\s \mapsto \R^d$ is defined as
\(  \label{def:Pi}
   \Pi_{t,x}[\gamma] := \dot\gamma(t-\tau_{t,x}[\gamma]).
\)

The system \eqref{eq:law}--\eqref{def:F} is equipped with the initial datum $\rho^0\in \mathcal{P}(\Omega_\s^0)$,
which is imposed in terms of the push-forward identity
\(  \label{MF:IC}
    \mathbb{I}\#\rho = \rho^0,
\)
with the mapping $\mathbb{I}: \Omega_\s^T \to \Omega_\s^0$ given by
\(  \label{def:I}
   \mathbb{I} : \gamma \mapsto \gamma|_{(-\infty,0]},
\)
and we recall that the push-forward measure $\mathbb{I}\#\rho\in \mathcal{P}(\Omega_\s^0)$ is defined by $\mathbb{I}\#\rho(B):= \rho(\mathbb{I}^{-1}(B))$,
where $B$ is any measurable subset of $\Omega_\s^0$.
Then, the initial conditions $x(0)$ and $v(0)$ for \eqref{eq:MF} are distributed according to $X_0\#\rho^0$
and, resp., $V_0\#\rho^0$, with the mappings
\(  \label{XtVt}
   X_t[\gamma] := \gamma(t), \qquad V_t[\gamma] := \dot\gamma(t).
\)

In Section \ref{sec:MF} we shall provide a proof of the following result
on existence and uniqueness of solutions of the mean-field system
\eqref{eq:law}--\eqref{MF:IC}.

\begin{theorem}[Existence and uniqueness of measure solutions] \label{thm:MF}
Fix $T>0$ and $\s<\c$.
Then, for any $\rho^0\in \mathcal{P}(\Omega_\s^0)$ there exists a unique probability measure
$\rho\in\POsT$ that verifies \eqref{eq:law}--\eqref{MF:IC}.
\end{theorem}

Moreover, we shall provide a stability result, controlling a suitable notion of distance of solutions
by the distance of the initial data. For this purpose, we equip the space $\POsT$ with the Monge--Kantorovich--Rubinstein distance
\(  \label{def:MKR}
   \mathcal{W}_T(\rho,\nu) := \inf_{\pi\in\Lambda(\rho,\nu)} \iint_{\Omega_\s^T \times \Omega_\s^T} \Norm{\gamma-\xi}_{\Omega_\s^T} \d\pi(\gamma,\xi),
\)
with $\Norm{\cdot}_{\Omega_\s^T}$ given by \eqref{def:norm} and $\Lambda(\rho,\nu)$ denoting the set of transference plans between the measures $\rho$ and $\nu$,
i.e., probability measures on the product space $\Omega_\s^T \times \Omega_\s^T$ with first and second marginals
$\rho$ and $\nu$, respectively.
We then have the following stability result.

\begin{theorem}[Stability] \label{thm:stability}
For each $T>0$ there exists a constant $M_T>0$ such that
for any $\rho^0$, $\nu^0 \in \POsZ$,
\(  \label{stability}
   \mathcal{W}_T(\rho,\nu) \leq M_T \mathcal{W}_0(\rho^0, \nu^0),
\)
where $\rho\in\POsT$, and, resp., $\nu\in\POsT$ are the unique solutions of \eqref{eq:law}--\eqref{MF:IC}
constructed in Theorem \ref{thm:MF} subject to the initial data $\rho^0$ and, resp. $\nu^0$.
\end{theorem}

The above stability theorem justifies the approximation of the system \eqref{eq:law}--\eqref{MF:IC}
by finite sets of discrete particles satisfying \eqref{eq:tau}--\eqref{IC:comp}.
Indeed, let us consider $(\solx, \solv)$ a solution of \eqref{eq:tau}--\eqref{IC:comp}
with a given $N\in\N$, and define $\rho_N\in\POsT$ the atomic measure being concentrated on the set of trajectories
$\solx = (x_1, \ldots, x_N)$, i.e.,
\(   \label{atomic}
   \rho_N(\gamma) := \frac{1}{N}  \sum_{j=1}^N \delta(\gamma - x_j),
\)
with $\delta$ denoting the Dirac measure on $\Omega_\s^T$, concentrated on the constant zero function $\gamma\equiv 0$.
Then, from \eqref{def:F} we have
\[
   F_t[\rho_N](x,v) = \frac{1}{N} \sum_{j=1}^N \psi\left( \left|x_j(t-\tau_{t,x}[x_j]) - x \right|\right) \left(\dot x_j(t-\tau_{t,x}[x_j]) - v \right),
\]
where $\tau_{t,x}[x_j]$ is the unique solution of \eqref{eq:tau_gamma} with $\gamma:=x_j$.
Consequently, \eqref{eq:tau}--\eqref{eq:CS2} can be rewritten as
\[
    \dot x_i(t) = v_i(t), \qquad \dot v_i(t) = \frac{N}{N-1} F_t[\rho_N](x_i(t),v_i(t)).
\]
Note that $\rho_N = \mbox{law}(\{x_i\}_{i\in[N]})$. Consequently,
up to the factor $\frac{N}{N-1}$ that can be removed by rescaling of time,
$\rho_N\in \POsT$ given by \eqref{atomic}
is a solution of \eqref{eq:law}--\eqref{MF:IC} subject to the initial datum
\(   \label{IC:atomic}
      \rho^0_N(\gamma) := \frac{1}{N}  \sum_{j=1}^N \delta(\gamma - x^0_j).
\)
Therefore, the important consequence of Theorem \ref{thm:stability} is that it provides
a method to derive the measure-valued mean-field description \eqref{eq:law}--\eqref{MF:IC}
as the limit $N\to\infty$ of the particle approximations \eqref{eq:tau}--\eqref{IC:comp}.

\begin{corollary}[Convergence of the particle method]  \label{cor:convergence}
Given $\rho^0\in \mathcal{P}(\Omega_\s^0)$, take a sequence $\rho^0_N\in \mathcal{P}(\Omega_\s^0)$
of atomic measures of the form \eqref{IC:atomic} such that
\[
    \lim_{N\to\infty} \mathcal{W}_0(\rho^0_N, \rho^0) = 0.
\]
Let $\rho^N\in\POsT$ be given by \eqref{atomic} with $(\solx, \solv)$
a solution of \eqref{eq:tau}--\eqref{IC:comp}.
Then, for any $T>0$,
\[
    \lim_{N\to\infty} \mathcal{W}_T(\rho_N, \rho) = 0,
\]
where $\rho^N\in\POsT$ is the unique solution of \eqref{eq:law}--\eqref{MF:IC}
subject to the initial datum $\rho^0$.
\end{corollary}

Finally, in Section \ref{subsec:FP} we discuss the question
whether one can express the mean-field limit problem \eqref{eq:law}--\eqref{MF:IC}
in terms of a Fokker-Planck equation for some time dependent phase-space particle density $g_t \in \mathcal{P}(\R^d\times\R^d)$, $t\geq 0$.
A short consideration leads to the conclusion that no such description is possible,
i.e., one indeed has to resort to the formulation in terms of probability measures
on the space of time-dependent trajectories.

%%%%%%%%%%%%%%%%%%%%%%%%%%%%%%%
\section{Global existence and uniqueness of solutions - proof of Theorem \ref{thm:ex}} \label{sec:ex}

We start by proving two auxiliary results establishing unique solvability of the equation \eqref{eq:tau}
and a bound on the delay $\tau_{ij}$.

\begin{lemma} \label{lem:tau}
For some $t\in\R$, let $x\in C_\s((-\infty,t]; \R^d)$ with the Lipschitz constant $\s<\c$.
Then, for each $z\in\R^d$, the equation
\(  \label{tauxz}
   \c \tau = |z-x(t-\tau)|
\)
is uniquely solvable in $\tau\geq 0$.
%and the solution $\tau$ satisfies
%\(  \label{bound:tau}  \frac{|z-x(t)|}{\c+\s} \leq \tau \leq \frac{|z-x(t)|}{\c-\s}. \)
\end{lemma}

\begin{proof}
See the proof of \cite[Lemma 2.2]{Has:sdHK}.
\end{proof}

\begin{lemma}\label{lem:tauij}
For some $t\in\R$, let $\solx = (x_1,\ldots, x_N) \in C_\s((-\infty,t]; \R^{Nd})$ with all trajectories $x_i$ uniformly Lipschitz continuous on $(-\infty,t]$
with Lipschitz constant $\s<\c$.
Then for all $i, j \in [N]$ we have
\(  \label{est:tauij}
   \tau_{ij}(t) \leq \frac{\dx(t)}{\c-\s},
\)
where $\tau_{ij}=\tau_{ij}(t)$ is the unique solution of \eqref{eq:tau} and $\dx=\dx(t)$ defined in \eqref{dXdV}.
\end{lemma}

\begin{proof}
By \eqref{eq:tau} we have
\[
    \c \tau_{ij} =   |\wxji - x_i|.
\]
On the other hand, due to the $\s$-Lipschitz continuity of $x_j$,
\(   \label{est:wxjixj}
   |\wxji - x_j|  = |x_j(t-\tau_{ij}) - x_j(t)| \leq \s \tau_{ij}.
\)
Therefore, by the triangle inequality,
\[
   \c \tau_{ij} = |\wxji - x_i| \leq |x_i-x_j| + |\wxji-x_j| \leq \dx(t) + \s\tau_{ij}, %= \dx + \frac{\s}{\c} |\wxji - x_i|,
\]
and \eqref{est:tauij} follows.
\end{proof}

Next, we prove local in time existence and uniqueness of solutions of the system \eqref{eq:tau}--\eqref{IC:comp}.
The proof is an adaptation of the Picard-Lindel\"of theorem, based on construction of a contraction mapping.

\begin{lemma}\label{lem:local}
Let the influence function $0 \leq \psi \leq 1$ be uniformly Lipschitz continuous on $[0,\infty)$.
Let $0<\s<\c$ be fixed and let the initial datum $\solv^0 = (v^0_1, \ldots, v^0_N) \in \mathbb{V}_\s^0$,
%with $S^0=\frac{\dx(0)}{\s-\c}$, with $\dx$ is defined in \eqref{dXdV}.
Then the system \eqref{eq:tau}--\eqref{IC:comp} admits unique local solutions
in the class of Lipschitz continuous velocity trajectories.
\end{lemma}

\def\Picard{\Upsilon}

\begin{proof}
Let us fix any $\m$ such that
\[
   \s < \m < \c
\]
and for some $T>0$ to be specified later, define the set $\mathbb{W}_\m^T$
of continuous velocity trajectories on $(-\infty,T]$,
which coincide with the initial datum $\solv^0$ on the interval $(-\infty, 0]$,
and are uniformly bounded by $\m$
and $2\m$-Lipschitz continuous on the interval $[0, T]$, i.e.,
\[
   \mathbb{W}_\m^T &:=& \Bigl\{ \solv\in C((-\infty, T];\R^{Nd});\, \solv|_{(-\infty, 0]} \equiv \solv^0, \; \solv|_{[0,T]}\in C_{2\m}([0,T]; R^{Nd}),  \Bigr. \\
       && \qquad\qquad\qquad\qquad\qquad\qquad\qquad \Bigl.  |v_i(t)| \leq \m \mbox{ for all } t\in [0,T],\, i\in [N]\Bigr\},
\]
%with $S^\m:=\frac{\dx(0)}{\m-\c}$, and the initial datum $v^0$ is extended from the interval $[-S^0,0]$ to $[-S^\m,0]$
%in an arbitrary (but fixed) way such that its continuity and the property $|v_i(t)| \leq \s$ are preserved for all $s\in [-S^\m,0]$
%and all $i\in [N]$. For instance, $v^0$ can be chosen to take the constant value $v^0(-S^0)$ on the interval $[-S^\m, -S^0]$.
where we again use the notation $\solv = (v_1, \ldots, v_N)\in \R^{Nd}$
and $C_{2\m}([0,T]; R^{Nd})$ denotes the set of $2\m$-Lipschitz continuous
trajectories on $[0,T]$. Let us note that since the initial datum $\solv^0$ verifies
$|v^0_i(t)|\leq \s < \m$ for all $t\leq 0$ and $i\in [N]$, the set $\mathbb{W}_\m^T$ is nonempty.

We equip the set $\mathbb{W}_\m^T$ with the topology of uniform convergence,
i.e., with the $L^\infty$-norm on $(0,T)$.
Then $\mathbb{W}_\m^T$ is a complete metric space. 

We construct a unique local solution of \eqref{eq:tau}--\eqref{IC:comp} by finding a unique fixed point
of the Picard operator $\Picard: \mathbb{W}_\m^T \to \mathbb{W}_\m^T$,
where the $i$-th component is the $d$-dimensional vector-valued function $\Picard_i[\omega](t)$,
given for $t\in [0,T]$ by
\( \label{Picard}
  \Picard_i[\omega](t) :=  v_i^0(0) + \frac{1}{N-1} \int_0^t \sum_{j\neq i} \psi(|\xi_j(s-\tau_{ij}(s))-\xi_i(s)|) (\omega_j(s-\tau_{ij}(s))-\omega_i(s)) \d s
  %\qquad \mbox{for }t\in [0,T],
\)
with
\(    \label{def:xi}
   \xi_i(t) = x_i^0(0) + \int_0^t \omega_i(s) \d s \qquad\mbox{for }i\in[N], \quad t\leq T,
\)
and we set $\Picard[\omega]$ equal to $\solv^0$ on $(-\infty,0]$.
%Here $\P_\s: \R^d \to \R^d$ is the projection in $\R^d$ on the ball $\{w\in\R^d; |w| \leq \s\}$.
Clearly, with $\omega\in \mathbb{W}_\m^T$, the trajectories $\xi_i$ defined by \eqref{def:xi} are Lipschitz continuous with Lipschitz constant $\m < \c$.
Then, Lemma \ref{lem:tau} gives unique solutions $\tau_{ij}=\tau_{ij}(t)$ of the equation
\( \label{Picard:tau}
   \c\tau_{ij}(t) = \left|\xi_j(t-\tau_{ij}(t)) - \xi_i(t) \right|,
\)
for all $i,j\in [N]$ and $t\leq T$.

We show that for sufficiently small $T>0$ the Picard operator $\Picard$ defined by \eqref{Picard} %is well defined,
maps the space $\mathbb{W}_\m^T$ into itself and is a contraction on $\mathbb{W}_\m^T$
with respect to the $L^\infty(0,T)$-norm.

With the global bound $\psi\leq 1$ we have for any $\omega\in \mathbb{W}_\m^T$,
\[
   \left| \Picard_i[\omega](t) \right| &\leq&  \left| v_i^0(0) \right| + \frac{1}{N-1} \int_0^t \sum_{j\neq i} \left| \omega_j(s-\tau_{ij}(s)) \right| + \left| \omega_i(s)) \right| \d s \\
   &\leq&  \s + 2T\m,
\]
so that choosing $T \leq \frac{\m - \s}{2\m}$ gives
\[
   \left| \Picard_i[\omega](t) \right| \leq \m
\]
for all $i\in [N]$ and $t \in [0,T]$.
Moreover, for any $0\leq t_1 < t_2 \leq T$ we have
\[
   \left| \Picard_i[\omega](t_1) - \Picard_i[\omega](t_2) \right| &\leq&
      \frac{1}{N-1} \int_{t_1}^{t_2} \sum_{j\neq i} \psi(|\xi_j(s-\tau_{ij}(s))-\xi_i(s)|) |\omega_j(s-\tau_{ij}(s))-\omega_i(s)| \d s \\
      &\leq&
      2\m |t_1-t_2|,
\]
so that $\Picard[\omega]$ is $2\m$-Lipschitz continuous.
Consequently, for small enough $T>0$, $\Picard$ maps the space $\mathbb{W}_\m^T$ into itself.

To prove contractivity of $\Picard$, let us pick $\omega, \kappa \in \mathbb{W}_\m^T$
and calculate, for any fixed $i\in [N]$,
\(
   \label{integrand}
   \left|\Picard_i[\omega](t) - \Picard_i[\kappa](t)\right| &\leq&  \frac{1}{N-1} \sum_{j\neq i}
       \int_0^t \bigl| \psi(|\xi_j(s-\tau^\xi_{ij}(s))-\xi_i(s)|)(\omega_j(s-\tau^\xi_{ij}(s))-\omega_i(s)) \bigr. \\
      &&\qquad\qquad \bigl. -  \psi(|\eta_j(s-\tau^\eta_{ij}(s))-\eta_i(s)|)(\kappa_j(s-\tau^\eta_{ij}(s)) - \kappa_i(s)) \bigr| \d s,
    \nonumber
\)
where
\[
   \xi_i(t) = x_i^0(0) + \int_0^t \omega_i(s) \d s, \qquad
   \eta_i(t) = x_i^0(0) + \int_0^t \kappa_i(s) \d s,
\]
and $\tau_{ij}^\xi$ and, resp., $\tau_{ij}^\eta$ are solutions of \eqref{Picard:tau} with $\xi$, resp., $\eta$.
Note that $\xi_i \equiv \eta_i$ on $(-\infty,0]$ for all $i\in[N]$ and
\(  \label{est:xi-eta}
    \Norm{\xi_i - \eta_i}_{L^\infty(0,T)} \leq \max_{t\in [0,T]} \left| \int_0^t \omega_i(s) - \kappa_i(s) \d s \right|
       \leq T \Norm{\omega_i-\kappa_i}_{L^\infty(0,T)}.
\)
We now estimate, for $s\in [0,T]$, the integrand in \eqref{integrand} by
\(   \nonumber
   \bigl| \psi(|\xi_j(s-\tau^\xi_{ij}(s))-\xi_i(s)|)(\omega_j(s-\tau^\xi_{ij}(s))-\omega_i(s)) 
         -  \psi(|\eta_j(s-\tau^\eta_{ij}(s))-\eta_i(s)|)(\kappa_j(s-\tau^\eta_{ij}(s))) - \kappa_i(s) \bigr| \\
         \label{est:integrand}
         \leq 
      \bigl| \psi(|\xi_j(s-\tau^\xi_{ij}(s))-\xi_i(s)|) - \psi(|\eta_j(s-\tau^\eta_{ij}(s))-\eta_i(s)|) \bigr|
         \bigl| \omega_j(s-\tau^\xi_{ij}(s))-\omega_i(s) \bigr| \\
         \nonumber
         + \psi(|\eta_j(s-\tau^\eta_{ij}(s))-\eta_i(s)|)
           \bigl| \omega_j(s-\tau^\xi_{ij}(s)) - \kappa_j(s-\tau^\eta_{ij}(s)) - \omega_i(s) + \kappa_i(s)  \bigr|.
\)
To estimate the first term we use the assumed uniform Lipschitz continuity of $\psi$ with constant $L_\psi$
and the uniform boundedness $|\omega_i| \leq \m$,
\[
   \bigl| \psi(|\xi_j(s-\tau^\xi_{ij}(s))-\xi_i(s)|) - \psi(|\eta_j(s-\tau^\eta_{ij}(s))-\eta_i(s)|) \bigr|
         \bigl| \omega_j(s-\tau^\xi_{ij}(s))-\omega_i(s) \bigr| \\
         \leq
         2\m L_\psi \bigl| |\xi_j(s-\tau^\xi_{ij}(s))-\xi_i(s)| - |\eta_j(s-\tau^\eta_{ij}(s))-\eta_i(s)| \bigr|.
\]
By the triangle inequality we have %, for $s\in [0,T]$,
\[
  \bigl| |\xi_j(s-\tau_{ij}^\xi(s))-\xi_i(s)| - |\eta_j(s-\tau_{ij}^\eta(s)) - \eta_i(s)| \bigr| &\leq&
       \Norm{\xi_i - \eta_i}_{L^\infty(0,T)} + |\xi_j(s-\tau_{ij}^\xi(s)) - \eta_j(s-\tau_{ij}^\eta(s))|  \\
       &\leq& 
       2 \Norm{\xi - \eta}_{L^\infty(0,T)} + |\eta_j(s-\tau_{ij}^\xi(s)) - \eta_j(s-\tau_{ij}^\eta(s))| \\
       &\leq& 
       2 \Norm{\xi - \eta}_{L^\infty(0,T)} + \m |\tau_{ij}^\xi(s) - \tau_{ij}^\eta(s)|,
\]
where we used the $\m$-Lipschitz continuity of $\eta_j$ for the last inequality.
Now, with \eqref{Picard:tau} we have
\(   \label{est:tautau}
     |\tau_{ij}^\xi(s) - \tau_{ij}^\eta(s)| = \c^{-1} \bigl| |\xi_j(s-\tau_{ij}^\xi(s)) - \xi_i(s)|
                - |\eta_j(s-\tau_{ij}^\eta(s)) - \eta_i(s)| \bigr|,
\)
so that
\[
   \bigl| |\xi_j(s-\tau_{ij}^\xi(s))-\xi_i(s)| - |\eta_j(s-\tau_{ij}^\eta(s)) - \eta_i(s)| \bigr| &\leq&
      2 \Norm{\xi - \eta}_{L^\infty(0,T)} 
      \\ &+& \m\c^{-1} \bigl| |\xi_j(s-\tau_{ij}^\xi(s)) - \xi_i(s)|
                - |\eta_j(s-\tau_{ij}^\eta(s)) - \eta_i(s)| \bigr|,
\]
which immediately gives
\(   \nonumber
    \bigl| |\xi_j(s-\tau_{ij}^\xi(s))-\xi_i(s)| - |\eta_j(s-\tau_{ij}^\eta(s)) - \eta_i(s)| \bigr|
        &\leq&  2 \left( 1 - \m\c^{-1} \right)^{-1}  \Norm{\xi - \eta}_{L^\infty(0,T)}   \\
        &\leq&  2 \left( 1 - \m\c^{-1} \right)^{-1}  T \Norm{\omega - \kappa}_{L^\infty(0,T)},
        \label{est:tautautau}
\)
where we used \eqref{est:xi-eta} for the second inequality.
Consequently, the first term in the right-hand side of \eqref{est:integrand} is bounded by the expression
\[
   4\m L_\psi \left( 1 - \m\c^{-1} \right)^{-1}  T \Norm{\omega - \kappa}_{L^\infty(0,T)}.
\]

The second term of the right-hand side in \eqref{est:integrand} is estimated as follows, using again $\psi\leq 1$,
\[
   \psi(|\eta_j(s-\tau^\eta_{ij}(s))-\eta_i(s)|)
           \bigl| \omega_j(s-\tau^\xi_{ij}(s)) - \kappa_j(s-\tau^\eta_{ij}(s)) - \omega_i(s) + \kappa_i(s)  \bigr|  \\
       \leq
           \Norm{\omega - \kappa}_{L^\infty(0,T)}
               + \bigl| \omega_j(s-\tau^\eta_{ij}(s)) - \kappa_j(s-\tau^\eta_{ij}(s)) \bigr| 
               + \bigl| \omega_j(s-\tau^\xi_{ij}(s)) - \omega_j(s-\tau^\eta_{ij}(s)) \bigr| \\
        \leq   2\Norm{\omega - \kappa}_{L^\infty(0,T)}
            + 2\m  |\tau_{ij}^\xi(s) - \tau_{ij}^\eta(s)|,
\]
where we used the $2\m$-Lipschitz continuity of $\omega$ in the second inequality.
Then, using \eqref{est:tautau} and \eqref{est:tautautau}, we conclude that
the second term in the right-hand side of \eqref{est:integrand} is bounded by the expression
\[
 %  \psi(|\eta_j(s-\tau^\eta_{ij}(s))-\eta_i(s)|) \bigl| \omega_j(s-\tau^\xi_{ij}(s)) - \kappa_j(s-\tau^\eta_{ij}(s)) - \omega_i(s) + \kappa_i(s)  \bigr|  \leq 
        2 \left(  1 + 2 \m\c^{-1} \left( 1 - \m\c^{-1} \right)^{-1}  T \right) \Norm{\omega - \kappa}_{L^\infty(0,T)}.
\]
Therefore, with \eqref{integrand} we finally arrive at
\(   \label{contractivity}
   \Norm{\Picard_i[\omega] - \Picard_i[\kappa]}_{L^\infty(0,T)} \leq
      2 T \left( 1 + 2\m (L_\psi + \c^{-1}) \left( 1 - \m\c^{-1} \right)^{-1}  T \right)
%       \left( 1 - \s\c^{-1} \right)^{-1}  \max\{S^0,T\} + 1 + L_v  \left( \c - \s \right)^{-1} \max\{S^0,T\}  \right)
       \Norm{\omega - \kappa}_{L^\infty(0,T)}
\)
and choosing $T>0$ sufficiently small, the claim follows.

Finally, let us note that choosing $\m_1\neq \m_2$ may lead to different values
of $T_1$, $T_2$, however, due to the contraction property \eqref{contractivity},
the solution remains unique on $[0, \min\{T_1, T_2\}]$.
\end{proof}

With Lemma \ref{lem:local} we constructed a unique solution of \eqref{eq:tau}--\eqref{IC:comp}
on a sufficiently short time interval $[0,T]$.
In the next step we prove that, assuming $\solv^0 \in \mathbb{V}_\s^0$,
the velocity trajectories remain uniformly bounded by $\s$,
which implies that the solution is in fact global in time.
For this purpose, let us define the velocity radius of the agent group,
\(  \label{def:Rv}
   R_v(t) := \max_{i\in[N]} |v_i(t)|.
\)
%The following lemma shows that the velocity radius is bounded uniformly in time by the radius of the initial datum, defined as
%\(  \label{def:Rv0} R_v^0 := \max_{t\in [-S^0,0]} R_v(t),\qquad \mbox{with } S^0=\frac{\dx(0)}{\s-\c}. \)

\begin{lemma}\label{lem:Rvbound}
Let the initial datum $\solv^0 \in \mathbb{V}_\s^0$ with $\s<\c$.
Then, along the solutions of \eqref{eq:tau}--\eqref{IC:comp}, the diameter $R_v$ defined in \eqref{def:Rv} satisfies
\(  \label{v-bound}
   %\max_{i\in[N]} |x_i(t)|
   R_v(t) \leq \s \qquad\mbox{for all } t\geq 0.
\)
%with $R_v^0$ given by \eqref{def:Rv0}.
\end{lemma}

\begin{proof}
Since $\s<\c$, we may fix $\eps >0$ such that $\s + \eps < \c$.
We shall prove that for all $t\geq 0$
\(  \label{R-eps}
   R_v(t) < \s + \eps.
\)
By continuity of the velocity trajectories,
\eqref{R-eps} holds on the maximal interval $[0,T)$ for some $T>0$.
For contradiction, let us assume that $T<+\infty$.
Then we have $R_v(T) = \s + \eps$.
Since $\s + \eps < \c$, Lemma \ref{lem:local} allows us to extend the solution
$(\solx, \solv)$ past $T$. I.e., the solution exists on some right neighborhood of $T$
and we have
\(  \label{R-cont}
   \tot{}{t+} R_v(T)^2 \geq 0,
\)
where $\tot{}{t+}  R_v(T)^2$ denotes the right-hand side derivative of $R_v^2$ at $t=T$.
Moreover, by continuity of the velocity trajectories,
there exists an index $i\in[N]$ such that $R_v(t) \equiv |v_i(t)|$
for $t\in [T,T+\delta)$ for some $\delta>0$.
From \eqref{eq:CS2} we then have
\[
   \tot{}{t+} R_v(T)^2 = \tot{}{t+} |v_i(T)|^2 &=& \frac{2}{N-1} \sum_{j=1}^N \widetilde\psi_{ij} \left[v_j(T-\tau_{ij}(T))- v_i(T)\right]\cdot v_i(T) \\
     &=& \frac{2}{N-1} \sum_{j=1}^N  \widetilde\psi_{ij} \left[ v_j(T-\tau_{ij}(T))\cdot v_i(T) - |v_i(T)|^2\right].
\]
We now distinguish two cases:
\begin{itemize}
\item
If $\tau_{ij}(T) = 0$ for all $j\in[N]$, then \eqref{eq:tau} implies $x_j(T) = x_i(T)$ for all $j\in[N]$.
Moreover, since by definition $|v_j(T)| \leq R_v(T) = |v_i(T)|$, we have by the Cauchy-Schwarz inequality
\(  \nonumber
   v_j(T-\tau_{ij}(T))\cdot v_i(T) - |v_i(T)|^2 &=& v_j(T)\cdot v_i(T) - |v_i(T)|^2   \\
   &\leq& \left( |v_j(T)| - |v_i(T)| \right) |v_i(T)| \leq 0,
   \label{CauchySchwarz1}
\)
so that
\[
   \tot{}{t+} R_v(T)^2 = \tot{}{t+} |v_i(T)|^2 
     \leq \frac{2}{N-1} \sum_{j=1}^N  \widetilde\psi_{ij} \left( |v_j(T)| - |v_i(T)| \right) |v_i(T)| \leq 0.
\]
Then \eqref{R-cont} implies $\tot{}{t+} R_v(T)^2 = 0$, which means that equality takes place
in the Cauchy-Schwarz inequality \eqref{CauchySchwarz1}, i.e.,
\[
   v_j(T)\cdot v_i(T) = |v_i(T)|^2 \qquad\mbox{for all } j\in [N].
\]
That means that $v_j(T) = v_i(T)$ for all $j\in[N]$, and since also $x_j(T) = x_i(T)$ for all $j\in[N]$,
the system reached equilibrium at time $T$ and does not evolve further
(note that Lemma \ref{lem:local} provides local uniqueness,
so the constant equilibrium solution is the only possible continuation
past the time $T$).
This trivially implies that \eqref{R-eps} holds for all $t>0$.

\item
If there exists at least one $j\in[N]$ such that $\tau_{ij}(T)>0$, then again by the Cauchy-Schwarz inequality we have
\[  %\label{CauchySchwarz2}
   v_j(T-\tau_{ij}(T))\cdot v_i(T) - |v_i(T)|^2 \leq \left( |v_j(T-\tau_{ij}(T))| - |v_i(T)| \right) |v_i(T)| < 0,
\]
since $|v_j(T-\tau_{ij}(T))| \leq R_v(T-\tau_{ij}(T)) < \s + \eps$
and $|v_i(T)| = R_v(T) = \s + \eps > 0$.
Consequently, %\eqref{CauchySchwarz2} implies that
\[
   \tot{}{t+} R_v(T)^2 < 0,
\]
which is a contradiction to \eqref{R-cont}.
\end{itemize}

The uniform bound \eqref{v-bound} is obtained by taking the limit $\eps\to 0$.
\end{proof}

Combination of the local uniqueness and existence result of Lemma \ref{lem:local}
with the uniform bound on the velocity diameter directly implies
the global existence and uniqueness claim of Theorem \ref{thm:ex}.

It remains to observe that, by Lemma \ref{lem:tauij}, we have
\[
   t - \tau_{ij}(t) \geq t - \frac{\dx(t)}{\c-\s},
\]
and from \eqref{eq:CS1},
\[
   \dx(t) \leq \dx(0) + \int_0^t \dv(s) \d s \leq \dx(0) + 2\s t,
\]
where we used the bound $\dv(t) \leq 2R_v(t) \leq 2\s$ provided by \eqref{v-bound}.
Consequently,
\[
   t - \tau_{ij}(t) \geq \frac{(\c-3\s)t - \dx(0)}{\c-\s},
\]
%and if $3\s\leq \c$, then
and
\[
   \min_{t\in [0,T]} \left( t - \tau_{ij}(t) \right) \geq - \frac{\dx(0) + [\c-3\s]^- T}{\c-\s}.
\]
%Therefore, if $3\s\leq \c$, the initial datum only needs to be prescribed on the bounded interval $\left[- \frac{\dx(0)}{\c-\s}, 0\right]$.
Therefore, the initial datum is only relevant on the bounded interval $\left[- S(T), 0\right]$,
with $S(T)$ given by \eqref{ST}.

%%%%%%%%%%%%%%%%%%%%%%%%%%
\section{Flocking - proof of Theorem \ref{thm:flocking}}\label{sec:flocking}
In this Section we shall assume the influence function $\psi$ to be a nonincreasing function of its argument on $[0,\infty)$.
This is without loss of generality, since if $\psi$ was not monotone,
it could be replaced by its nonincreasing rearrangement $\Psi$ in the below proofs,
\(  \label{Psi}
   \Psi(u) := \min_{s\in[0,u]} \psi(s) \qquad\mbox{for } u\geq 0,
\)
without affecting their validity.

\begin{lemma}\label{lem:shrinkage}
Along the solutions of \eqref{eq:tau}--\eqref{IC:comp} we have
\(   \label{claim:shrinkage}
    \tot{}{t} \dv(t) \leq  - \frac{N}{N-1}\upsi(t) \dv(t) + 2D(t)  \qquad\mbox{for almost all } t>0,
\)
with $\dv=\dv(t)$ given by \eqref{dXdV},
\(   \label{def:upsi}
   \upsi(t) := \min_{i,j \in [N]} \widetilde\psi_{ij} (t),
\)
and
\(   \label{def:D}
   D(t) :=  \max_{i\in [N]} \frac{1}{N-1}  \sum_{j\neq i}  \widetilde\psi_{ij} |\wvji - v_j|.
\)
\end{lemma}

\begin{proof}
Due to the continuity of the solution trajectories $x_i(t)$,
there is an at most countable system of open, mutually disjoint
intervals $\{\mathcal{I}_\sigma\}_{\sigma\in\N}$ such that
\[
   \bigcup_{\sigma\in\N} \overline{\mathcal{I}_\sigma} = [0,\infty)
\]
and for each ${\sigma\in\N}$ there exist indices $i(\sigma)$, $k(\sigma)$
such that
\[
   \dv(t) = |v_{i(\sigma)}(t) - v_{k(\sigma)}(t)| \quad\mbox{for } t\in \mathcal{I}_\sigma.
\]
Then, using the abbreviated notation $i:=i(\sigma)$, $k:=k(\sigma)$,
we have for every $t\in \mathcal{I}_\sigma$,
\(  \label{shrinkage:1}
   \frac12 \tot{}{t} \dv(t)^2 &=& \frac12 \tot{}{t} |v_i - v_k|^2 \\
    &=& \frac{1}{N-1} \sum_{j\neq i} \widetilde\psi_{ij} \bigl(\wvji - v_i\bigr)\cdot (v_i-v_k)
      - \frac{1}{N-1} \sum_{j\neq k} \widetilde\psi_{kj} \bigl(\wv{j}{k} - v_k\bigr)\cdot (v_i-v_k).   \nonumber
\)
Let us work on the first term of the right-hand side. We have for any $j\in [N]$,
\(   \label{shrinkage:2}
   \widetilde\psi_{ij} \left(\wvji - v_i\right)\cdot (v_i-v_k) = \widetilde\psi_{ij} \left(\wvji - v_j\right)\cdot (v_i-v_k) + \widetilde\psi_{ij} \left(v_j- v_i\right)\cdot (v_i-v_k).
\)
By the Cauchy-Schwarz inequality and \eqref{def:D} we have
\[
   \frac{1}{N-1} \sum_{j\neq i} \widetilde\psi_{ij} \left(\wvji - v_j\right)\cdot (v_i-v_k) \leq 
   \frac{1}{N-1} \sum_{j\neq i} \widetilde\psi_{ij} \left|\wvji - v_j\right| \left| v_i-v_k \right| \leq D(t) \dv(t).
\]
For the second term in \eqref{shrinkage:2}, we observe, using the Cauchy-Schwarz inequality,
\[
   (v_j- v_i)\cdot (v_i-v_k) &=& (v_j-v_k)\cdot(v_i-v_k) - |v_i-v_k|^2 \\
      &\leq& |v_i-v_k| \bigl( |v_j-v_k| - |v_i-v_k| \bigr) \leq 0,
\]
since, by definition, $|v_j-v_k| \leq \dx = |v_i-v_k|$.
Moreover, with \eqref{def:upsi} we have
\[
    \widetilde \psi_{ij} \left(v_j- v_i\right)\cdot (v_i-v_k) \leq \upsi \, (v_j- v_i)\cdot (v_i-v_k).
\]
Carrying out analogous steps for the second term of the right-hand side of \eqref{shrinkage:1},
we finally obtain
\[
   %\frac12 \tot{}{t} |x_i - x_k|^2
    \frac12 \tot{}{t} \dv^2 &\leq& 2\dv D + \frac{\upsi}{N-1} \sum_{j=1}^N 
      \bigl[(v_j- v_i) - (v_j-v_k)\bigr]\cdot (v_i-v_k) \\
      &=&  \left[ 2D - \frac{N}{N-1} \upsi \dv \right] \dv.
\]
This immediately gives the statement.
\end{proof}

%\cJH{from here constant initial datum:}

For $t\geq 0$ let us define the maximal delay at time $t$,
\(   \label{otau}
   \otau(t) := \max_{i,j\in [N]} \tau_{ij}(t).
\)

\begin{lemma} \label{lem:D}
Let the initial datum $\solv^0\in C((-\infty,0]; \R^{Nd})$ be Lipschitz continuous
on $(-\infty,0]$, with Lipschitz constant $L_{\solv}^0\geq 0$.
Then, along the solutions of \eqref{eq:tau}--\eqref{IC:comp}, we have
\(  \label{D}
   D(t) \leq  L_{\solv}^0 \, [t-\otau(t)]^- + \int_{[t-\otau(t)]^+}^t D(s) \d s + \int_{[t-\otau(t)]^+}^t \dv(s) \d s \qquad\mbox{for } t\geq 0,
\)
where we denoted $[a]^+ := \max\{a,0\}$, $[a]^- := \max\{-a,0\}$,
the quantity $D=D(t)$ is defined in \eqref{def:D} and $\dv=\dv(t)$ given by \eqref{dXdV}.
\end{lemma}

\begin{proof}
For any $i, j\in [N]$ with $i\neq j$ and $t\geq 0$ we have,
\[
   |\wvji - v_j| = |v_j(t-\tau_{ij}) - v_j(t)| &\leq& \int_{t-\tau_{ij}}^t \left| \dot v_j(s) \right| \d s  \\
       &\leq& \int_{t-\otau}^{[t-\otau]^+} \left| \dot v_j(s) \right| \d s + \int_{[t-\otau]^+}^t \left| \dot v_j(s) \right| \d s.
\]
For the first integral of the right-hand side we have
\[
    \int_{t-\otau}^{[t-\otau]^+} \left| \dot v_j(s) \right| \d s \leq L_{\solv}^0 \, [t-\otau]^-,
\]
while for the second we use \eqref{eq:CS2},
\[
    \int_{[t-\otau]^+}^t \left| \dot v_j(s) \right| \d s
      &\leq& \frac{1}{N-1} \int_{[t-\otau]^+}^t \sum_{\ell\neq j} \widetilde{\psi}_{j\ell} \left| \wtv^j_\ell(s) - v_j(s) \right| \d s \\
      &\leq& \frac{1}{N-1} \int_{[t-\otau]^+}^t \sum_{\ell\neq j} \widetilde{\psi}_{j\ell}  \left( \left| \wtv^j_\ell(s) - v_\ell(s) \right| + \left| v_\ell(s) - v_j(s) \right| \right) \d s \\
      &\leq&  \int_{[t-\otau]^+}^t D(s) \d s +  \int_{[t-\otau]^+}^t \dv(s) \d s,
\]
where we used the estimate $\frac{1}{N-1} \sum_{\ell\neq j} \widetilde{\psi}_{j\ell} \leq 1$ implied by universal bound $\psi\leq 1$.
Consequently,
\[
      D(t) &=&  \max_{i\in [N]} \frac{1}{N-1}  \sum_{j\neq i}  \widetilde\psi_{ij} |\wvji - v_j|   \\
      &\leq& \max_{i\in [N]} \frac{1}{N-1}  \sum_{j\neq i}  \widetilde\psi_{ij} \left( L_{\solv}^0 \, [t-\otau]^- + \int_{[t-\otau]^+}^t D(s) \d s +  \int_{[t-\otau]^+}^t \dv(s) \d s  \right)  \\
      &\leq& L_{\solv}^0 \, [t-\otau(t)]^- + \int_{[t-\otau]^+}^t D(s) \d s +  \int_{[t-\otau]^+}^t \dv(s) \d s.
\]
\end{proof}

\begin{lemma}  \label{lem:crazy}
Fix $\dx(0)$, $\dv(0)$, $L_{\solv}^0 \geq 0$, $\s>0$, and $\bb>0$ provided by assumption \eqref{ass:bb}.
Assume that there exist $\c>\s$, $\cc> D(0)$ and $\dd > \dv(0)$ such that   % \cJH{we actually need $\cc > D(0)$ and $\dd > \dv(0)$,  but $D(0)=0$ and $\dv(0)=\s$ for constant initial $v$}
\(   \label{crazy:1}
   L_{\solv}^0\taust + (\cc+\dd) \frac{e^{\bb\taust} - 1}{\bb} \leq \cc \qquad\mbox{with } \taust:= (\c - \s)^{-1} \left( \dx(0) + \frac{\dd}{\bb} \right),
   % ATTENTION; Here $\frac{\dd}{\bb}$ is correct (it is not $\frac{\dv(0)}{\bb}$) !!
\)
%\cJH{requires $\taust<1$}
and
\(   \label{crazy:2}
    %\frac{2\cc}{\psist-\bb} \leq (\dd - \dv(0))
    %\dv(0)
       \frac{2\cc}{\psist-\bb} \leq \dd - \dv(0)
     \qquad\mbox{with }
        %\psist := %\frac{N}{N-1} \psi\left( \frac{\c}{\c-\s} \left( \dx(0) + \frac{\dd}{\bb} \right) \right) > \bb.
        \psist := \psi(\c\taust) > \eta.
\)
%\cJH{requires $\bb<\psist$}
Then
\(   \label{claim:crazy}
   \dv(t) < \dd e^{-\bb t} \mbox{ and } D(t) < \cc e^{-\bb t}
\)
for all $t>0$.
\end{lemma}

\begin{proof}
Let us define the set
\(   \label{setS}
   \mathcal{S} := \left\{ t>0; \; D(t) < \cc e^{-\bb t} \quad\mbox{and}\quad \dv(t) < \dd e^{-\bb t} \right\}.
\)
%and note that for constant initial velocity trajectories we have $D(0)=0$.
Since, by assumption, $\cc>D(0)$ and $\dd > \dv(0)$, the set $\mathcal{S}$ is nonempty, so that we may define
$T:=\sup\mathcal{S}$.
We aim at showing that $T=+\infty$. For contradiction, let us assume that $T<+\infty$.
By continuity of the functions $\dv=\dv(t)$ and $D=D(t)$ we then have
\(  \label{for_contr}
   D(T) =  \cc e^{-\bb T} \qquad\mbox{or}\qquad  \dv(T) = \dd e^{-\bb T}.
\)
By the definition \eqref{dXdV} of the diameters $\dx$ and $\dv$ we have
\[
    \dx(t) \leq \dx(0) + \int_0^t \dv(s) \d s, % < \dx(0) + \frac{\dd}{\bb},
\]
and from the definition \eqref{setS} of the set $\mathcal{S}$ it follows for $t<T$,
\[
   \int_0^t \dv(s) \d s < \dd \int_0^t e^{-\bb s} \d s < \frac{\dd}{\bb},
\]
so that
\(   \label{est:dx}
   \dx(t) < \dx(0) + \frac{\dd}{\bb}.
\)
An application of Lemma \eqref{lem:tauij} then gives
\(   \label{est:otau}
   \otau(t) = \max_{i,j\in [N]} \tau_{ij}(t) \leq \frac{\dx(t)}{\c-\s} < \frac{1}{\c-\s} \left( \dx(0) + \frac{\dd}{\bb} \right) = \taust
\)
for $t<T$.
%Consequently, with the definition \eqref{crazy:1} of $\taust$,
%\[  \taust:= (\c - \s)^{-1} \left( \dx(0) + \frac{\dd}{\bb} \right),  \]

Using \eqref{D} and \eqref{setS}, we have for $t\leq T$,
\[
   D(t) &\leq& L_{\solv}^0 \, [t-\otau(t)]^- + \int_{[t-\otau(t)]^+}^t D(s) \d s + \int_{[t-\otau(t)]^+}^t \dv(s) \d s \\
      &<& L_{\solv}^0 \, [t-\taust]^- + (\cc  + \dd) \int_{[t-\taust]^+}^t e^{-\bb s} \d s  \\
      &\leq& \left( L_{\solv}^0 \, [t-\taust]^- e^{\bb t} + (\cc + \dd) \frac{e^{\bb\taust} - 1}{\bb}  \right) e^{-\bb t}.
%    &<& (\cc + \dd) \frac{e^{\bb\taust} - 1}{\bb}.
\]
An elementary calculation, using the fact that $\eta<1$ and realizing that \eqref{crazy:1} implies $\taust<1$, gives
\[
   [t-\taust]^- e^{\bb t} \leq \taust \qquad\mbox{for all } t\geq 0.
\]
Consequently,
\[
   D(T) < \left( L_{\solv}^0 \, \taust + (\cc + \dd) \frac{e^{\bb\taust} - 1}{\bb}  \right) e^{-\bb T},
\]
and assumption \eqref{crazy:1} gives $D(T) < \cc e^{-\bb T}$,
so that the first alternative in \eqref{for_contr} is excluded.
%\cJH{From here: $(\cc + \dd) \frac{e^{\bb\taust} - 1}{\bb} \leq \cc$}.

To exclude the second alternative in \eqref{for_contr}, we recall that $\widetilde\psi_{ij}(t) = \psi(|\wxji - x_i|)$
and apply the triangle inequality, \eqref{est:wxjixj} and \eqref{est:tauij}, which yields
\[
   |\wxji - x_i| \leq |\wxji - x_j| + |x_i - x_j| \leq \s \tau_{ij} + \dx(t) \leq \frac{\c}{\c-\s} \dx(t).
\]
Since, by assumption, $\psi$ is a non-increasing function, we have for all $t\geq 0$,
\[
    \widetilde\psi_{ij}(t) \geq \psi\left( \frac{\c}{\c-\s} \dx(t) \right).
\]
Combining this with \eqref{est:dx}, we obtain $\widetilde\psi_{ij}(t) \geq \psist$ for all $t<T$ and all $i,j\in [N]$,
with $\psist$ defined in \eqref{crazy:2}.
%\[ \psist := %\frac{N}{N-1}  \psi\left( \frac{\c}{\c-\s} \left(\dx(0) + \frac{\dd}{\bb} \right) \right).  \]
Then, Lemma \ref{lem:shrinkage} gives
\[
   \tot{}{t} \dv(t) &\leq&  - \frac{N}{N-1}\psist \dv(t) + 2D(t) \\
      &<&  - \psist \dv(t) + 2 \cc e^{-\bb t}
\]
for almost all $t<T$. An integration on the interval $(0,T)$ then yields
\[
   \dv(T) &\leq& \dv(0) e^{-\psist T} + 2 \cc \frac{e^{-\bb T} - e^{-\psist T}}{\psist - \bb} \\
      &<& \left( \dv(0) +  \frac{2\cc}{\psist-\bb} \right) e^{-\bb T},
\]
where we used the fact that $\bb<\psist$ in the second inequality.
Assumption \eqref{crazy:2} gives $\dv(T) < \dd e^{-\bb T}$,
and, consequently, the second alternative in \eqref{for_contr} is excluded.

We conclude that $T=+\infty$ which finishes the proof.
\end{proof}

Obviously, \eqref{claim:crazy} implies flocking in the sense of Definition \ref{def:flocking},
noting that the uniform boundedness of $\dx=\dx(t)$ follows from
\[
    \dx(t) \leq \dx(0) + \int_0^t \dv(s) \d s < \dx(0) + \dd \int_0^t e^{-\bb s} \d s < \dx(0) + \frac{\dd}{\bb}
    \qquad\mbox{for all } t\geq 0.
\]
%\cJH{If not constant IC: only relevant on bounded time $[-\taust,0]$.}

Clearly, the crucial step of the proof of Theorem \ref{thm:flocking}
is to find values of $\c>\s$, $\dd>\dv(0)$ and $\cc>D(0)$
such that the assumptions of Lemma \ref{lem:crazy} are satisfied.
This essentially means to resolve the nonlinear algebraic system \eqref{crazy:1}--\eqref{crazy:2}.
We were only able to find an analytical solution under the additional assumption
that the initial velocity trajectories $\solv^0$ are constant on $(-\infty,0]$.
Then $L_{\solv}^0 = 0$ and $D(0)=0$, which facilitates the following result.

\begin{lemma} \label{lem:notsocrazy}
Assume that $L_{\solv}^0 = 0$ and $D(0)=0$.
For any fixed $\dx(0)$, $\dv(0)\geq 0$, $\s>0$, and $\bb>0$
provided by assumption \eqref{ass:bb}, there exist $\cst>\s$,
$\dd>\dv(0)$ and $\cc>0$ such that for any $\c\geq \cst$ %, there exist $\dd>\dv(0)$ and $\cc>0$ such that
the conditions \eqref{crazy:1} and \eqref{crazy:2} are verified.
\end{lemma}

\begin{proof}
Assumption \eqref{ass:bb} and continuity of $\psi$ implies that, for fixed $\bb>0$, there exist
$\eps>0$ and $\dd > \dv(0)$ such that
\(   \label{crazy:4}
   \psi\left( (1+\eps) \left(\dx(0) + \frac{\dd}{\bb} \right) \right) > \bb.
\)
We set
\(   \label{crazy:5}
   \cc := \frac{\dd-\dv(0)}{2} \left[ \psi\left( (1+\eps) \left(\dx(0) + \frac{\dd}{\bb} \right) \right) - \bb \right].
\)
Now, with the values of $\s>0$, $\bb>0$, $\dd > \dv(0)$ and $\cc>0$ fixed, the equation
\(   \label{crazy:3}
   \frac{1}{\bb} \ln\left( \frac{\bb\cc}{\cc+\dd} + 1 \right) = \frac{1}{\c-\s} \left(\dx(0) + \frac{\dd}{\bb} \right)
\)
is uniquely solvable in $\c>\s$, since $\frac{1}{\c-\s}$ is a monotonically decreasing function of $\c>\s$
with values in $(0,\infty)$.
We denote this unique solution $\c_1$ and set
\(   \label{cst}
   \cst := \max\left\{ \c_1, \frac{1+\eps}{\eps} \s \right\}.
\)
We claim that the above choice of $\dd$, $\cc$ and $\cst$
verifies the assumptions of Lemma \ref{lem:crazy},
in particular, the conditions \eqref{crazy:1} and \eqref{crazy:2}.
Indeed, recalling the definition of $\taust$ given by \eqref{crazy:1},
\(   \label{taust}
   \taust = \frac{1}{\cst - \s} \left( \dx(0) + \frac{\dd}{\bb} \right),
\)
equation \eqref{crazy:3} gives
\[
    \frac{1}{\bb} \ln\left( \frac{\bb\cc}{\cc+\dd} + 1 \right) \leq \taust,
\]
which is equivalent to
\[
    (\cc+\dd) \frac{e^{\bb\taust} - 1}{\bb} \leq \cc. % \frac{\cc}{\cc + \dd}.
\]
This proves \eqref{crazy:1} with $L_{\solv}^0=0$.

By \eqref{cst} we have $\frac{\cst}{\cst-\s} \leq 1+\eps$, which implies
\[
   \cst \taust  % =\frac{\c}{\cst - \s} \left( \dx(0) + \frac{\dd}{\bb} \right)
      \leq (1+\eps) \left( \dx(0) + \frac{\dd}{\bb} \right).
\]
%with $\taust$ defined in \eqref{taust}. 
Due to the monotonicity of $\psi$, we then have for $\psist := \psi(\cst \taust)$
\[
   \psist \geq \psi\left( (1+\eps) \left( \dx(0) + \frac{\dd}{\bb} \right) \right),
\]
and, due to \eqref{crazy:4}, $\psist > \bb$.
Moreover,
\[
    \frac{2\cc}{\psist-\bb} \leq  2\cc \left[ \psi\left( (1+\eps) \left(\dx(0) + \frac{\dd}{\bb} \right) \right) - \bb \right]^{-1},
\]
and substituting for $\cc$ from \eqref{crazy:5} gives
\[
    \frac{2\cc}{\psist-\bb} \leq \dd - \s,
\]
which is \eqref{crazy:2}.

Finally, let us note that for any $\c\geq \cst$ we have
\[
   %\c \taust = \frac{\c}{\c - \s} \left( \dx(0) + \frac{\dd}{\bb} \right) < \frac{\cst}{\cst - \s} \left( \dx(0) + \frac{\dd}{\bb} \right),
      \frac{1}{\c - \s} \left( \dx(0) + \frac{\dd}{\bb} \right) \leq \taust.
\]
%since $\frac{1}{\c - \s}$ is a decreasing function of $\c$.
Consequently, \eqref{crazy:1} remains valid with $\cst$ replaced by $\c$.
%$\psi(\c\taust) \geq \psist$ and the proof of Lemma \ref{lem:crazy} remains valid.
Moreover,
\[
   \c \tau := \frac{\c}{\c - \s} \left( \dx(0) + \frac{\dd}{\bb} \right) \leq \frac{\cst}{\cst - \s} \left( \dx(0) + \frac{\dd}{\bb} \right) = \cst\taust,
\]
since $\frac{\c}{\c - \s}$ is a decreasing function of $\c>\s$.
Consequently, $\psi(\c\tau) \geq \psist$ and \eqref{crazy:2} keeps its validity as well.  % (with $\c$)
%the flocking result provided by Theorem \ref{thm:flocking} holds for any $\c\geq \cst$.
\end{proof}

\begin{remark} \label{rem:v0}
Note that for a given influence function $\psi$ and values of $\bb$, $\dv(0$), $\dx(0)$ and $\s$,
the proof of Lemma \ref{lem:notsocrazy} provides a recipe for calculating the value of $\cst$ explicitly
(although it may not provide the optimal value).
On the other hand, admitting nonconstant initial datum $\solv^0$
makes it necessary to decide the solvability of the algebraic system \eqref{crazy:1}--\eqref{crazy:2},
which does not seem to be achievable analytically.
However, the problem is well approachable by numerical methods.
\end{remark}

Finally, we provide the proof of Proposition \ref{corr:CS}.
Clearly, \eqref{ass:bb} is satisfiable whenever $\inf_{s\in(0,\infty)} \psi(s) > 0$.
Therefore, let us consider $\psi=\psi(s)$ such that $\lim_{s\to\infty} \psi(s) = 0$.
Moreover, recall that we assume $\psi=\psi(s)$ to be monotone
(otherwise we can replace it by its nonincreasing rearrangement).
Assumption \eqref{corr:CS:cond} implies that $\psi(s) \geq c s^\alpha$  for some constant $c>0$ and all $s>0$ large enough.
Denoting $\Psi(s):=c s^\alpha$, we readily calculate for $\alpha>-1$,
\[
   \lim_{\bb\to 0+}  \Psi\left(\dx(0) + \frac{\dv(0)}{\bb}\right) - \bb = 0,
   \qquad
   \tot{}{\bb} \left[ \Psi\left(\dx(0) + \frac{\dv(0)}{\bb}\right) - \bb \right]_{\bb=0+} = -\infty.
\]
Consequently, for $\bb>0$ small enough
\[
    \psi\left(\dx(0) + \frac{\dv(0)}{\bb}\right) - \bb \geq  \Psi\left(\dx(0) + \frac{\dv(0)}{\bb}\right) - \bb > 0,
\]
and the claim of Proposition \ref{corr:CS} follows.

%%%%%%%%%%%%%%%%%%%%%%%%%%%%%%%%%%%%%%%%%%%%%%%%%%%%
\section{Mean field limit} \label{sec:MF}
%%%%%%%%%%%%%%%%%%%%%%%%%%%%%%%%%%%%%%%%%%%%%%%%%%%%

Let us fix $T>0$ and $0<\s<\c$ and recall the definition \eqref{def:OmegasT} of the set
\[   %\label{def:OmegasT}
   \Omega_\s^T := %\left\{ \gamma\in C^1((-\infty,T]);\; \gamma=\gamma(t) \mbox{ is $\s$-Lipschitz continuous on $(-\infty, T]$} \right\}.
      \Bigl\{ \gamma\in C_b^1((-\infty, T];\R^d) \cap C_\s((-\infty,T]; \R^d); \;
         \dot\gamma|_{[0,T]} \in C_{2\s}([0,T]; \R^d) \Bigr\},
\]
%Obviously, $\Omega_\s^T$ is a subset of the space of $\R^d$-valued $C^1$-functions on $(-\infty,T]$
%that are globally Lipschitz continuous and their first-order derivative
%is Lipschitz continuous on $[0,T]$.
equipped with the norm $\Norm{\cdot}_{\Omega_\s^T}$ defined in \eqref{def:norm},
\[
   \Norm{\gamma}_{\Omega_\s^T}:= \Norm{\gamma}_{L^\infty(-\infty,T)} + \Norm{\dot \gamma}_{L^\infty(0,T)} \qquad
   \mbox{for } \gamma\in \Omega_\s^T.
\]
Obviously, due to the uniform boundedness of $\gamma\in\Omega_\s^T$ on $(-\infty,T)$ and the uniform Lipschitz continuity of $\dot\gamma$ on $(0,T)$,
we have $\Norm{\gamma}_{\Omega_\s^T} < +\infty$ for all $\gamma\in\Omega_\s^T$.
We also recall that $\POsT$ denotes the space of probability measures on $\Omega_\s^T$
with finite first-order moment
and it is equipped with the Monge--Kantorovich--Rubinstein distance $\mathcal{W}_T$
defined in \eqref{def:MKR}.
%\(  \label{def:MKR} \mathcal{W}_T(\rho,\nu) := \inf_{\pi\in\Lambda(\rho,\nu)} \iint_{\Omega_\s^T \times \Omega_\s^T} \Norm{\gamma-\xi}_{\Omega_\s^T} \d\pi(\gamma,\xi), \)
%where $\Lambda(\rho,\nu)$ denotes the set of transference plans between the measures $\rho$ and $\nu$,
%i.e., probability measures on the product space $\Omega_\s^T \times \Omega_\s^T$ with first and second marginals $\rho$ and $\nu$, respectively.
For any $\rho$, $\nu\in\POsT$ we have $\mathcal{W}_T(\rho,\nu) < +\infty$ since
\[
      \mathcal{W}_T(\rho,\nu) &\leq& \iint_{\Omega_\s^T \times \Omega_\s^T} \Norm{\gamma-\xi}_{\Omega_\s^T} \d\rho(\gamma)\d\nu(\xi) \\
         &\leq& \int_{\Omega_\s^T}  \Norm{\gamma}_{\Omega_\s^T} \d\rho(\gamma) + \int_{\Omega_\s^T} \Norm{\xi}_{\Omega_\s^T} \d\nu(\xi),
\]
and the measures $\rho$, $\nu\in\POsT$ have, by definition, finite first-order moment.
The set $\POsT$ endowed with the distance $\mathcal{W}_T$ is a complete metric space.

We consider the problem
\(
   \rho &=& \mbox{law}(x), \label{eq:MF0} \\
    x(t) &=& x(0) + \int_0^t v(s) \d s, \label{eq:MF1} \\
    v(t) &=& v(0) + \int_0^t F_s[\rho](x(s),v(s)) \d s   \label{eq:MF2}
\)
with the operator $F_t[\rho]: \R^d\times\R^d \to \R^d$ given by \eqref{def:F}, i.e.,
\[
   F_t[\rho](x,v) := \int_{\Omega^T_\s} \psi\left( \left|\Gamma_{t,x}[\gamma] - x \right|\right) \left(\Pi_{t,x}[\gamma] - v \right)  \d\rho(\gamma).
\]
Note that $F_t[\rho]$ is well defined since with the universal bound $\psi\leq 1$ and $\rho\in\POsT$ we have
\[
   \left| \int_{\Omega^T_\s} \psi\left( \left|\Gamma_{t,x}[\gamma] - x \right|\right) \left(\Pi_{t,x}[\gamma] - v \right)  \d\rho(\gamma)  \right|
   \leq
   \int_{\Omega^T_\s}  \left( \left| \Pi_{t,x}[\gamma] \right| + |v|  \right) \d\rho(\gamma) \leq \s + |v|.
\]
We recall that the initial datum $\rho^0\in \mathcal{P}(\Omega_\s^0)$ for the mean-field problem
is imposed in terms of the push-forward identity $\mathbb{I}\#\rho = \rho^0$, 
with the mapping $\mathbb{I}: \Omega_\s^T \to \Omega_\s^0$ given by \eqref{def:I}.
Then, $x(0)$ and, resp., $v(0)$ in \eqref{eq:MF1}-\eqref{eq:MF2} are distributed according to $X_0\#\rho^0$
and, resp., $V_0\#\rho^0$.

%%%%%%%%%%%%%%%%%%%%%%%%%%%%
\subsection{Auxiliary results}

For any $\rho\in \POsT$ we define the mapping $Z[\rho]: \Omega_\s^0 \to C^1((-\infty,T])$,
\(   \label{def:Z}
   Z[\rho]: \gamma^0 \mapsto \gamma,
\)
where $\gamma\in C^1((-\infty,T])$ is identical to $\gamma^0$ on $(-\infty,0]$
and on $(0,T]$ it is the unique solution $x=x(t)$ of the system \eqref{eq:MF1}--\eqref{eq:MF2}
subject to the initial datum $x(0)=\gamma^0(0)$ and $v(0)=\dot\gamma^0(0)$.
The existence and uniqueness of solutions of \eqref{eq:MF1}--\eqref{eq:MF2}
is established by a slight modification of the arguments carried out in Section \ref{sec:ex}.

\begin{lemma}\label{lem:Z}
For any $T>0$ and $\rho\in \POsT$, the mapping $Z[\rho]$ given by \eqref{def:Z} maps the set $\Omega^0_\s$
into $\Omega^T_\s$.
\end{lemma}

\begin{proof}
Let us fix $\gamma^0\in \Omega^0_\s$ and denote $x:= Z[\rho](\gamma^0)$, $v:=\dot x$.
We first prove that $x$ is $\s$-Lipschitz continuous on $(-\infty, T]$.
On the interval $(-\infty,0]$ this follows directly from the definition.
For $t\geq 0$ we have from \eqref{eq:MF2}--\eqref{def:F},
\[
   \frac12 \tot{}{t} |v(t)|^2 &=& \int_{\Omega^T_\s} \psi\left( \left|\Gamma_{t,x(t)}[\gamma] - x(t) \right|\right) \left( \Pi_{t,x(t)}[\gamma] - v(t) \right) \cdot v(t) \,  \d\rho(\gamma)  \\
      &\leq& \int_{\Omega^T_\s} \left( \left| \Pi_{t,x(t)}[\gamma] \right| - |v(t)| \right)  |v(t)| \,  \d\rho(\gamma),
\]
where we used the universal bound $\psi\leq 1$.
Moreover, by the definition \eqref{def:OmegasT}, we have $|\Pi_{t,x}[\gamma]| = |\dot\gamma(t-\tau_{t,x}[\gamma])| \leq \s$
for all $t\in (-\infty,T]$ and $\gamma\in\Omega^T_\s$.
Consequently, for almost all $t>0$,
\[
   \tot{}{t} |v(t)| \leq \s - |v(t)|,
\]
with $|v(0)| \leq \s$.
This implies $|\dot x(t)| = |v(t)| \leq \s$ for all $t>0$. %(see, e.g., my BLMS paper).
Since, by definition, $Z[\rho](\gamma^0) \equiv x$ on $(0,T]$,
we obtain the $\s$-Lipschitz continuity of $Z[\rho](\gamma^0)$ on the interval $[0,T]$.

Moreover, with the universal bound $\psi\leq 1$, $|v(t)| \leq \s$ and $\left| \Pi_{t,x}[\gamma] \right| \leq \s$, we have for $t\in [0,T]$
\[
    |\dot v(t)| = |F_s[\rho](x(t),v(t))|  \leq
         \int_{\Omega^T_\s} \psi\left( \left|\Gamma_{t,x}[\gamma] - x(t) \right|\right) \left| \Pi_{t,x}[\gamma] - v(t) \right|  \d\rho(\gamma)
         \leq 2\s.
\]
This gives the uniform $2\s$-Lipschitz continuity of $v= \tot{}{t} Z[\rho](\gamma^0)$ on $[0,T]$
and we finally conclude that $Z[\rho](\gamma^0)\in \Omega^T_\s$.
\end{proof}

\begin{lemma}  \label{lem:gamma-gamma}
For any $\gamma$, $\xi\in \Omega_\s^T$ and $x$, $y\in\R^d$, $t\in [0,T]$ we have
\(  \label{GammaGamma}
   \left|\Gamma_{t,x}[\gamma] - \Gamma_{t,y}[\xi] \right| \leq  %\left( 1 - \s\c^{-1} \right)^{-1}
        \frac{\c}{\c-\s} \Bigl( \Norm{ \gamma - \xi }_{L^\infty(0,t)} + \s\c^{-1} |x-y|  \Bigr),
\)
where the mapping $\Gamma_{t,x}[\gamma]$ is defined in \eqref{def:Gamma}, and
\(  \label{PiPi}
   \left|\Pi_{t,x}[\gamma] - \Pi_{t,y}[\xi] \right| \leq  %\left( 1 - \s\c^{-1} \right)^{-1}
      \Norm{ \dot\gamma - \dot\xi }_{L^\infty(0,t)}  +
        \frac{2\s}{\c-\s} \Bigl( \Norm{ \gamma - \xi }_{L^\infty(0,t)} + |x-y|  \Bigr),
  %      \frac{L_{\dot\xi}}{\c-\s} \Bigl( \Norm{ \gamma - \xi }_{L^\infty(0,T)} + |x-y|  \Bigr),
\)
with $\Pi_{t,x}[\gamma]$ defined in \eqref{def:Pi}.
\end{lemma}

\begin{proof}
Fix $t\in [0,T]$.
By definition of $\Gamma_{t,x}$, we have
\[
   \left|\Gamma_{t,x}[\gamma] - \Gamma_{t,y}[\xi] \right| =
      \left| \gamma(t-\tau_{t,x}[\gamma]) - \xi(t-\tau_{t,y}[\xi]) \right|,
\]
with
\(  \label{taugammaxi}
   \c\tau_{t,x}[\gamma] = \bigl| x - \gamma(t-\tau_{t,x}[\gamma] \bigr|, \qquad
   \c\tau_{t,y}[\xi] = \bigl| y - \xi(t-\tau_{t,y}[\xi] \bigr|.
\)
The triangle inequality gives %, for $s\in [0,T]$,
\[
   \bigl| \gamma(t-\tau_{t,x}[\gamma]) - \xi(t-\tau_{t,y}[\xi]) \bigr| &\leq&
     \bigl| \gamma(t-\tau_{t,x}[\gamma])  - \xi(t-\tau_{t,x}[\gamma]) \bigr| 
     +  \bigl| \xi(t-\tau_{t,x}[\gamma])  - \xi(t-\tau_{t,y}[\xi]) \bigr|   \\
     &\leq&
     \Norm{ \gamma - \xi }_{L^\infty(0,t)} + \s \bigl| \tau_{t,x}[\gamma] - \tau_{t,y}[\xi] \bigr|,
\]
where we used the $\s$-Lipschitz continuity of $\xi$ for the last inequality.
Now, with \eqref{taugammaxi} we have
\[   %   \label{est:tautau}
    \bigl| \tau_{t,x}[\gamma] - \tau_{t,y}[\xi] \bigr|
         &=& \c^{-1} \bigl| |\gamma(t-\tau_{t,x}[\gamma]) - x| - |\xi(t-\tau_{t,y}[\xi]) - y| \bigr|  \\
         &\leq& \c^{-1} \Bigl( |x-y| + \bigl| \gamma(t-\tau_{t,x}[\gamma]) - \xi(t-\tau_{t,y}[\xi]) \bigr| \Bigr),
\]
so that
\[
   \bigl| \gamma(t-\tau_{t,x}[\gamma]) - \xi(t-\tau_{t,y}[\xi]) \bigr| &\leq&
      \Norm{ \gamma - \xi }_{L^\infty(0,t)} + \s\c^{-1} \Bigl( |x-y| + \bigl| \gamma(t-\tau_{t,x}[\gamma]) - \xi(t-\tau_{t,y}[\xi]) \bigr| \Bigr),
\]
which immediately gives
\[
   \bigl| \gamma(t-\tau_{t,x}[\gamma]) - \xi(t-\tau_{t,y}[\xi]) \bigr| \leq
       %\left( 1 - \s\c^{-1} \right)^{-1}  
       \frac{\c}{\c-\s} \Bigl( \Norm{ \gamma - \xi }_{L^\infty(0,t)} + \s\c^{-1} |x-y|  \Bigr).
\]
This proves \eqref{GammaGamma}.

Moreover,
\[
   \left|\Pi_{t,x}[\gamma] - \Pi_{t,y}[\xi] \right| &=&
      \left| \dot\gamma(t-\tau_{t,x}[\gamma]) - \dot\xi(t-\tau_{t,y}[\xi]) \right|    \\
    &\leq&
      \left| \dot\gamma(t-\tau_{t,x}[\gamma]) - \dot\xi(t-\tau_{t,x}[\gamma]) \right| 
        +  \left| \dot\xi(t-\tau_{t,x}[\gamma]) - \dot\xi(t-\tau_{t,y}[\xi]) \right|    \\
    &\leq&
        \Norm{ \dot\gamma - \dot\xi }_{L^\infty(0,t)} + %L_{\dot\xi} \bigl| \tau_{t,x}[\gamma] - \tau_{t,y}[\xi] \bigr|.
           2\s \bigl| \tau_{t,x}[\gamma] - \tau_{t,y}[\xi] \bigr|,
\]
where we used the fact that $\dot\xi$ is $2\s$-Lipschitz continuous on $[0,T]$.
Combining the above estimates, we readily obtain
\[
   \bigl| \tau_{t,x}[\gamma] - \tau_{t,y}[\xi] \bigr| \leq \frac{1}{\c-\s} \bigl( \Norm{ \gamma - \xi }_{L^\infty(0,t)} + |x-y| \bigr),
\]
which gives \eqref{PiPi}.
\end{proof}

\begin{lemma}\label{lem:LipschZ}
For any $T>0$ there exists a constant $L_Z^T>0$ such that, for any $\nu\in\POsT$,
\(  \label{LipschZ}
   \Norm{(Z[\nu](\eta_0) - Z[\nu](\xi_0)}_{\Omega_\s^T} \leq L_Z^T  \Norm{\eta_0-\xi_0}_{\Omega_\s^0}
   \qquad\mbox{for all }\eta_0, \xi_0 \in \Omega_\s^0.
\)
%for all $\eta_0$, $\xi_0\in \Omega_\s^0$.
The value of the constant $L_Z^T$ is explicitly calculable and depends only on $\c$, $\s$, $T$ and the Lipschitz constant $L_\psi$
of the influence function $\psi$.
\end{lemma}

\begin{proof}
Let us fix $\nu\in\POsT$ and $\eta_0$, $\xi_0\in \Omega_\s^0$
and denote $x:=Z[\nu](\eta_0)|_{[0,T]}$, $v:=\dot x$, and $y:=Z[\nu](\xi_0)|_{[0,T]}$, $w:=\dot y$.
Then we have, by the definition \eqref{def:norm} of $\Norm{\cdot}_{\Omega_\s^T}$,
\[
   \Norm{(Z[\nu](\eta_0) - Z[\nu](\xi_0)}_{\Omega_\s^T}  =
      \Norm{\eta_0-\xi_0}_{L^\infty(-\infty,0)} + \Norm{x-y}_{L^\infty(0,T)} + \Norm{v-w}_{L^\infty(0,T)}.
\]
With
\[
   \Norm{x-y}_{L^\infty(0,T)} &\leq& |x(0) - y(0)| +  T\Norm{v-w}_{L^\infty(0,T)}   \\
      &\leq& \Norm{\eta_0-\xi_0}_{\Omega_\s^0} +  T\Norm{v-w}_{L^\infty(0,T)},
\]
we have
\(    \label{goingback}
    \Norm{(Z[\nu](\eta_0) - Z[\nu](\xi_0)}_{\Omega_\s^T} \leq 2 \Norm{\eta_0-\xi_0}_{\Omega_\s^0} + (1+T) \Norm{v-w}_{L^\infty(0,T)}.
\)
We thus need to estimate the term $\Norm{v-w}_{L^\infty(0,T)}$.
For any $t\in (0,T)$ we calculate, using \eqref{eq:MF2},
\[
   |v(t) - w(t)| &\leq& \int_0^t |F_s[\nu](x(s),v(s)) - F_s[\nu](y(s),w(s))| \, \d s \\
     &\leq&
        \int_0^t %   \int_0^t
         \int_{\Omega^T_\s} \left| \psi\left( \left|\Gamma_{s,x(s)}[\gamma] - x(s) \right|\right) \left(\Pi_{s,x(s)}[\gamma] - v(s) \right)
         \right. \\
         && \qquad\qquad\qquad \left.
         -  \psi\left( \left|\Gamma_{s,y(s)}[\gamma] - y(s) \right|\right) \left(\Pi_{s,y(s)}[\gamma] - w(s) \right) \right|  \d\nu(\gamma) \d s.
\]
We estimate the integrand as
\(
   \left| \psi\left( \left|\Gamma_{s,x(s)}[\gamma] - x(s) \right|\right) \left(\Pi_{s,x(s)}[\gamma] - v(s) \right)
         -  \psi\left( \left|\Gamma_{s,y(s)}[\gamma] - y(s) \right|\right) \left(\Pi_{s,y(s)}[\gamma] - w(s) \right) \right|   \nonumber \\
      \leq
         \left| \psi\left( \left|\Gamma_{s,x(s)}[\gamma] - x(s) \right|\right) - \psi\left( \left|\Gamma_{s,y(s)}[\gamma] - y(s) \right|\right) \right|
              \left| \Pi_{s,x(s)}[\gamma] - v(s) \right|   \label{est:ZZZ} \\
           +   \left| \Pi_{s,x(s)}[\gamma] - \Pi_{s,y(s)}[\gamma] \right| + |v(s) - w(s)|,   \nonumber
\)
where we used the universal bound $\psi\leq 1$.
For the first term of the right hand side we apply the bound $|v(s)|\leq\s$ provided by Lemma \ref{lem:Z}
and the bound $|\Pi_{s,x(s)}[\gamma]| \leq \s$ implied by the fact that $\gamma\in\Omega_\s^T$, so that we have
\[
    \left| \psi\left( \left|\Gamma_{s,x(s)}[\gamma] - x(s) \right|\right) - \psi\left( \left|\Gamma_{s,y(s)}[\gamma] - y(s) \right|\right) \right|
              \left| \Pi_{s,x(s)}[\gamma] - v(s) \right| \\
              \leq 
              2\s L_\psi \left( \left| \Gamma_{s,x(s)}[\gamma] - \Gamma_{s,y(s)}[\gamma] \right| + |x(s)-y(s)| \right),   
\]
where $L_\psi$ is the Lipschitz constant of the influence function $\psi$.
The estimate \eqref{GammaGamma} of Lemma \ref{lem:gamma-gamma} gives
\[
   \left| \Gamma_{s,x(s)}[\gamma] - \Gamma_{s,y(s)}[\gamma] \right| \leq \frac{\s}{\c-\s} |x(s) - y(s)|,
\]
se that the first term of the right hand side in \eqref{est:ZZZ} is bounded by
\[
    \frac{2\s\c L_\psi}{\c-\s} |x(s)-y(s)|.
\]
For the second term of the right-hand side of \eqref{est:ZZZ} we apply the estimate \eqref{PiPi} of Lemma \ref{lem:gamma-gamma},
\[
   \left| \Pi_{s,x(s)}[\gamma] - \Pi_{s,y(s)}[\gamma] \right| \leq \frac{2\s}{\c-\s} |x(s)-y(s)|.
\]
Combining the above estimates, we arrive at
\[
   |v(t) - w(t)| &\leq& \int_0^t \frac{2\s}{\c-\s} \left( \c L_\psi + 1 \right) |x(s)-y(s)| + |v(s) - w(s)| \d s  \\
      &\leq& \frac{2\s T}{\c-\s} \left( \c L_\psi + 1 \right) \left( |x(0)-y(0)| + \int_0^t |v(s)-w(s)| \d s \right) + \int_0^t |v(s) - w(s)| \d s  \\
      &\leq& \frac{2\s T}{\c-\s}\left( \c L_\psi + 1 \right) |x(0)-y(0)| + \left( \frac{2\s T}{\c-\s} \left( \c L_\psi + 1 \right) + 1 \right) \int_0^t |v(s)-w(s)|,
\]
for all $t\in (0,T)$, where we used the estimate
\[
   \int_0^t |x(s)-y(s)| \d s &\leq& T |x(0)-y(0)| + \int_0^t \int_0^s |v(\sigma)-w(\sigma)| \d\sigma \d s \\
     &\leq& T |x(0)-y(0)| + T \int_0^t |v(s)-w(s)| \d s. 
\]
Denoting
\[
   K_T := \frac{2\s T}{\c-\s}\left( \c L_\psi + 1 \right) \exp\left( \frac{2\s T^2}{\c-\s} \left( \c L_\psi + 1 \right) + T \right),
\]
an application of the Gronwall lemma on $(0,T)$ yields
\[
   \Norm{v-w}_{L^\infty(0,T)} \leq K_T |x(0)-y(0)| \leq K_T \Norm{\eta_0-\xi_0}_{\Omega_\s^0}.
\]
Recalling \eqref{goingback}, we finally obtain
\[
   \Norm{(Z[\nu](\eta_0) - Z[\nu](\xi_0)}_{\Omega_\s^T} \leq \bigl[ 2 + (1+T) K_T \bigr]  \Norm{\eta_0-\xi_0}_{\Omega_\s^0},
\]
so that the claim \eqref{LipschZ} follows with the choice $L_Z^T := 2 + (1+T) K_T$.
Obviously, the constant $L_Z^T$ does not depend on the choice of $\nu\in\POsT$.
\end{proof}

\begin{lemma}\label{lem:LipschZZ}
Fix $T>0$. Then for any $\nu\in\POsT$ we have
\(  \label{LipschZZ}
   \mathcal{W}_T(Z[\nu]\#\rho^0, Z(\nu)\#\mu^0) \leq L_Z^T \mathcal{W}_0(\rho^0, \mu^0)
   \qquad\mbox{for all }\rho^0, \mu^0 \in \POsZ,
\)
where $L_Z^T$ is the constant provided by Lemma \ref{lem:LipschZ}.
\end{lemma}

\begin{proof}
A straightforward adaptation of the proof of \cite[Lemma 3.13]{CCR},
using the Lipschitz continuity result \eqref{LipschZ} of Lemma \ref{lem:LipschZ}.
\end{proof}

\begin{lemma} \label{lem:Wdist}
For any pair of Borel measurable mappings $Y_1, Y_2: \Omega_\s^0 \to \Omega_\s^T$ and $\mu^0\in \POsZ$, we have
\(
   \mathcal{W}_T(Y_1\#\mu^0, Y_2\#\mu^0) \leq
        \sup_{\gamma^0\in \Omega_\s^0} \Norm{Y_1(\gamma^0)-Y_2(\gamma^0)}_{\Omega_\s^T}.
        %\sup_{\gamma^0\in \supp(\rho^0)} \Norm{Y(\gamma^0)-Z(\gamma^0)}_{\Omega_\s^T}.
\)
\end{lemma}

\begin{proof}
See the proof of \cite[Lemma 3.11]{CCR}.
\end{proof}

\begin{lemma}\label{lem:Zgronwall}
There exists a continuous positive function $r:\R^+\to\R^+$ such that for any $\mu^0\in \POsZ$
and any $\rho$, $\nu\in\POsT$,
\(   \label{Zgronwall}
   \mathcal{W}_T(Z[\rho]\#\mu^0, Z(\nu)\#\mu^0) \leq r(T) \mathcal{W}_T(\rho, \nu).
\)
The function $r=r(T)$ is explicitly calculable and depends only on the values of $\s$, $\c$
and the Lipschitz constant $L_\psi$ of the influence function $\psi$.
Moreover, it has the property
\(  \label{lim-r}
   \lim_{T\to 0} r(T) = 0.
\)
\end{lemma}

\begin{proof}
Let us fix $\rho$, $\nu\in \POsT$
and apply Lemma \ref{lem:Wdist},
\[
   \mathcal{W}_T(Z[\rho]\#\mu^0, Z[\nu]\#\mu^0)
       \leq  \sup_{\gamma^0\in \Omega_\s^0} \Norm{(Z[\rho](\gamma^0) - Z[\nu](\gamma^0)}_{\Omega_\s^T}.
    % \sup_{t\in(-\infty,T]} W_1(X_t\#(Z[\rho]\#\mu^0), X_t\#(Z[\nu]\#\mu^0))   \\
    %&\leq&  \sup_{t\in(0,T]} \Norm{ Z[\rho](t) - Z[\nu](t) }_{L^\infty(\mbox{supp}(\mu^0))},
    %\qquad\cJH{\mbox{supp}(\mu^0)?}
\]
Let us now fix $\gamma^0\in \Omega_\s^0$  %$\cJH{\mbox{supp}(\mu^0)}$
and denote $x:=Z[\rho](\gamma^0)|_{[0,T]}$, $v:=\dot x$ and $y:=Z[\nu](\gamma^0)|_{[0,T]}$, $w:=\dot y$.
Then, according to the definition \eqref{def:norm} of the norm $\Norm{\cdot}_{\Omega_\s^T}$, we have
\[
   \Norm{(Z[\rho](\gamma^0) - Z[\nu](\gamma^0)}_{\Omega_\s^T} =
       \Norm{x-y}_{L^\infty(0,T)} + \Norm{v-w}_{L^\infty(0,T)},
\]
where we used the fact that, by definition, $Z[\rho](\gamma^0)$ and $Z[\nu](\gamma^0)$
coincide on $(-\infty,0]$.
Moreover, with $|x(t)-y(t)| \leq \int_0^t |\dot x(s) - \dot y(s)| \d s = \int_0^t |v(s) - w(s)| \d s$, we have
\[
   \Norm{(Z[\rho](\gamma^0) - Z[\nu](\gamma^0)}_{\Omega_\s^T} \leq (1+T) \Norm{v-w}_{L^\infty(0,T)}.
\]
Let us fix $t\in (0,T)$ and calculate, using \eqref{eq:MF2},
\[
   |v(t) - w(t)| &\leq& \int_0^t |F_s[\rho](x(s),v(s)) - F_s[\nu](y(s),w(s))| \, \d s \\
     &=&
       \int_0^t
       \left|
         \int_{\Omega^s_\s} \psi\left( \left|\Gamma_{s,x(s)}[\gamma] - x(s) \right|\right) \left(\Pi_{s,x(s)}[\gamma] - v(s) \right)  \d\rho(\gamma)
         \right. \\
      && \qquad\qquad\qquad \left.
       -  \int_{\Omega^s_\s} \psi\left( \left|\Gamma_{s,y(s)}[\xi] - y(s) \right|\right) \left(\Pi_{s,y(s)}[\xi] - w(s) \right)  \d\nu(\xi)
       \right| \d s.
\]
Let $\pi$ be an optimal transference plan between the measures $\rho$ and $\nu$
(for the proof of existence of the optimal transference plan, see, e.g., \cite[Theorem 4.1]{old-new}).
Then, using the fact that $\pi$ has marginals $\rho$ and $\nu$, we have
\[
      && \int_{\Omega^s_\s} \psi\left( \left|\Gamma_{s,x(s)}[\gamma] - x(s) \right|\right) \left(\Pi_{s,x(s)}[\gamma] - v(s) \right)  \d\rho(\gamma)  \\
      && \qquad
       -  \int_{\Omega^s_\s} \psi\left( \left|\Gamma_{s,y(s)}[\xi] - y(s) \right|\right) \left(\Pi_{s,y(s)}[\xi] - w(s) \right)  \d\nu(\xi) \\
       &=& \iint_{\Omega^s_\s \times \Omega^s_\s}
          \Bigl[ \psi\left( \left|\Gamma_{s,x(s)}[\gamma] - x(s) \right|\right) \left(\Pi_{s,x(s)}[\gamma] - v(s) \right)
          \Bigr. \\
        && \qquad\qquad\qquad \Bigl.
          - \psi\left( \left|\Gamma_{s,y(s)}[\xi] - y(s) \right|\right) \left(\Pi_{s,y(s)}[\xi] - w(s) \right)
            \Bigr] \d\pi(\gamma,\xi).
\]
We write the integrand as
\(
   \label{integrand:MF}
   \begin{aligned}
   & \Bigl[ \psi\left( \left|\Gamma_{s,x(s)}[\gamma] - x(s) \right|\right) - \psi\left( \left|\Gamma_{s,y(s)}[\xi] - y(s) \right|\right) \Bigr]
      \left(\Pi_{s,x(s)}[\gamma] - v(s) \right) \\
      &+   \psi\left( \left|\Gamma_{s,y(s)}[\xi] - y(s) \right|\right)
      \Bigl[  \left(\Pi_{s,x(s)}[\gamma] - v(s) \right) - \left(\Pi_{s,y(s)}[\xi] - w(s) \right) \Bigr].
      \end{aligned}
\)
The first line in \eqref{integrand:MF} is estimated from above by
\[
   && \Bigl| \psi\left( \left|\Gamma_{s,x(s)}[\gamma] - x(s) \right|\right) - \psi\left( \left|\Gamma_{s,y(s)}[\xi] - y(s) \right|\right) \Bigr|
      \left|\Pi_{s,x(s)}[\gamma] - v(s) \right| \\
      &\leq&
      2\s L_\psi \Bigl| \left|\Gamma_{s,x(s)}[\gamma] - x(s) \right| - \left|\Gamma_{s,y(s)}[\xi] - y(s) \right| \Bigr| \\
      &\leq&
      2\s L_\psi \Bigl( \left|\Gamma_{s,x(s)}[\gamma] - \Gamma_{s,y(s)}[\xi] \right| + \left|x(s) - y(s) \right|  \Bigr) \\
      &\leq&
      2\s L_\psi \left[ \frac{\c}{\c-\s} \Bigl( \Norm{ \gamma - \xi }_{L^\infty(0,s)} + \s\c^{-1} |x(s)-y(s)|  \Bigr) + \left|x(s) - y(s) \right| \right]  \\
      &=&
      \frac{2\s\c L_\psi}{\c-\s} \left[ \Norm{ \gamma - \xi }_{L^\infty(0,s)} + \left|x(s) - y(s) \right| \right],
\]
where we used the Lipschitz continuity of the influence function $\psi$
with constant $L_\psi$ and the estimate $\left|\Pi_{s,x(s)}[\gamma] - v(s) \right| \leq 2\s$ for the first inequality,
and estimate \eqref{GammaGamma} of Lemma \ref{lem:gamma-gamma} for the last inequality.

The second line in \eqref{integrand:MF} is estimated by
\[
   &&  \psi\left( \left|\Gamma_{s,y(s)}[\xi] - y(s) \right|\right)
      \Bigl|  \left(\Pi_{s,x(s)}[\gamma] - v(s) \right) - \left(\Pi_{s,y(s)}[\xi] - w(s) \right) \Bigr|  \\
    &\leq&
        |v(s)-w(s)|  +  \Norm{ \dot\gamma - \dot\xi }_{L^\infty(0,s)}  +
          \frac{2\s}{\c-\s} \Bigl( \Norm{ \gamma - \xi }_{L^\infty(0,s)} + |x(s)-y(s)|  \Bigr),
\]
where we used the universal bound $\psi\leq 1$ and the estimate \eqref{PiPi} of Lemma \ref{lem:gamma-gamma}.

Combining the above estimates, we arrive at
\[
   |v(t)-w(t)| &\leq& \int_0^t \iint_{\Omega^s_\s \times \Omega^s_\s}
            \left[ \frac{2\s}{\c-\s}\left( \c L_\psi + 1\right) \Norm{ \gamma - \xi }_{L^\infty(0,s)} + \Norm{\dot\gamma - \dot\xi }_{L^\infty(0,s)} \right]
              \d\pi(\gamma,\xi) \d s + \\
      && \qquad\qquad\qquad +\, \int_0^t |v(s)-w(s)| + \frac{2\s}{\c-\s}\left( \c L_\psi + 1 \right) |x(s)-y(s)| \d s.
\]
Recalling that $x(0)=y(0)$, we have
\[
   \int_0^t |x(s)-y(s)| \d s \leq \int_0^t \int_0^s |v(\sigma)-w(\sigma)| \d\sigma \d s \leq  t \int_0^t |v(s)-w(s)| \d s,
\]
which gives
\[
   |v(t)-w(t)| &\leq&  \max\left\{ 1, \frac{2\s}{\c-\s}\left( \c L_\psi + 1\right) \right\} \int_0^t 
        \iint_{\Omega^s_\s \times \Omega^s_\s} \Norm{\gamma-\xi}_{\Omega_\s^s} \d\pi(\gamma,\xi) \d s \\
        &&\qquad\qquad\qquad +
          \left( \frac{2\s t}{\c-\s}\left( \c L_\psi + 1 \right) + 1 \right) \int_0^t |v(s)-w(s)| \d s.
\]
Since $\pi$ is an optimal transference plan, we have $\iint_{\Omega^s_\s \times \Omega^s_\s} \Norm{\gamma-\xi}_{\Omega_\s^s} \d\pi(\gamma,\xi) = \mathcal{W}_s(\rho,\nu)$,
so that
\[
   |v(t)-w(t)| \leq  \max\left\{ 1, \frac{2\s}{\c-\s}\left( \c L_\psi + 1\right) \right\} \int_0^t  \mathcal{W}_s(\rho,\nu) \d s
      + \left( \frac{2\s t}{\c-\s}\left( \c L_\psi + 1 \right) + 1 \right) \int_0^t |v(s)-w(s)| \d s.
\]
An application of the Gronwall lemma gives then
\[
   |v(t)-w(t)| \leq \bar r(t) \int_0^t \mathcal{W}_s(\rho,\nu) \d s
\]
with
\(   \label{barr}
   \bar r(t) := \max\left\{ 1, \frac{2\s}{\c-\s}\left( \c L_\psi + 1\right) \right\}  \exp\left(\frac{2\s t}{\c-\s}\left( \c L_\psi + 1 \right) + 1 \right).
\)
We finally conclude that
\( 
   \Norm{(Z[\rho](\gamma^0) - Z[\nu](\gamma^0)}_{\Omega_\s^T} &\leq& (1+T) \Norm{v-w}_{L^\infty(0,T)}  \nonumber \\
       &\leq& (1+T) \bar r(T) \int_0^T \mathcal{W}_s(\rho,\nu) \d s \label{ZZZZ} \\
       &\leq& (1+T) T \bar r(T) \mathcal{W}_T(\rho,\nu),  \nonumber
\)
where for the last inequality we used the fact that, by definition, $\mathcal{W}_s(\rho,\nu)$
is a nondecreasing function of $s\in[0,T]$.
Consequently, \eqref{Zgronwall} is verified with $r(T):= (1+T) T \bar r(T)$,
and \eqref{lim-r} holds since $\bar r(T)$ is bounded as $T\to 0$.
\end{proof}

%%%%%%%%%%%%%%%%%%%%%%%
\subsection{Existence and uniqueness of solutions}

For a fixed initial datum $\rho^0\in \POsZ$ we define the set $\POsTI \subset \POsT$,
\(  \label{def:POsTI}
   \POsTI := \left\{ \rho\in \POsT;\; \mathbb{I}\#\rho = \rho^0 \right\}.
\)
The solution $\rho=\mbox{law}(x)$ of the mean-field problem \eqref{eq:MF1}--\eqref{eq:MF2}
shall be constructed as the unique fixed point of the mapping
$\ZZ[\rho^0]: \POsTI \to \POsTI$,
\(  \label{def:ZZ}
    \ZZ[\rho^0](\rho) := Z[\rho]\#\rho^0,
\)
with the mapping $Z: \Omega^0_\s \to \Omega^T_\s$ defined in \eqref{def:Z}.

\begin{lemma}\label{lem:ZZ}
Fix $\rho^0\in \POsZ$.
For small enough $T>0$, the mapping $\ZZ[\rho^0]$ given by \eqref{def:ZZ} is a contraction on $\POsTI$ in the topology
induced by the Monge-Kantorowich-Rubinstein distance $\mathcal{W}_T$ defined in \eqref{def:MKR}.
\end{lemma}

\begin{proof}
We apply the Banach contraction theorem on the complete metric space $(\POsTI, \mathcal{W}_T)$.
Clearly, by Lemma \ref{lem:Z}, $\ZZ[\rho^0]$ maps $\POsTI$ into itself.
From Lemma \ref{lem:Zgronwall} we have
\[
   \mathcal{W}_T(\ZZ[\rho^0](\rho), \ZZ[\rho^0](\nu)) = \mathcal{W}_T(Z[\rho]\#\rho^0, Z[\nu]\#\rho^0)
       \leq r(T) \mathcal{W}_T(\rho, \nu),
\]
where $r=r(T)$ is a continuous function with $\lim_{T\to 0} r(T) = 0$.
Consequently, $\ZZ[\rho^0]$ is a contraction on $\POsTI$ for small enough $T>0$.
\end{proof}

The solution $\rho=\ZZ[\rho^0](\rho)$ constructed in Lemma \ref{lem:ZZ}
for short times $T>0$ can be extended
as long as the characteristics given by \eqref{eq:MF1}--\eqref{eq:MF2}
remain $\s$-Lipschitz continuous with $2\s$-Lipschitz continuous derivative.
This property is guaranteed by Lemma \ref{lem:Z} for any $T>0$.
Consequently, the solutions are global in time, which concludes the proof of Theorem \ref{thm:MF}.

%%%%%%%%%%%%%%%%%%%%%%%%%%%%%%%%%
\subsection{Stability}

\begin{theorem}\label{thm:stability}
Fix $T>0$. Then there exists a constant $M_T>0$ such that
for any $\rho^0$, $\nu^0 \in \POsZ$,
\(  \label{stability}
   \mathcal{W}_T(\rho,\nu) \leq M_T \mathcal{W}_0(\rho^0, \nu^0).
\)
where $\rho=Z[\rho]\#\rho^0$ and $\nu=Z[\nu]\#\nu^0$ are the unique fixed points
of the mapping $\ZZ$ provided by Lemma \ref{lem:ZZ}.
\end{theorem}

\begin{proof}
Using the triangle inequality, we have for any $t\in (0,T)$,
\[
   \mathcal{W}_t(\rho, \nu) &=& \mathcal{W}_t(Z[\rho]\#\rho^0, Z(\nu)\#\nu^0) \\
    &\leq&   
       \mathcal{W}_t(Z[\rho]\#\rho^0, Z(\nu)\#\rho^0) + \mathcal{W}_t(Z[\nu]\#\rho^0, Z(\nu)\#\nu^0).
\]
The first term of the right-hand side is estimated using Lemma \ref{lem:Wdist} and \eqref{ZZZZ},
\[
   \mathcal{W}_t(Z[\rho]\#\rho^0, Z[\nu]\#\rho^0)
       &\leq&  \sup_{\gamma^0\in \Omega_\s^0} \Norm{(Z[\rho](\gamma^0) - Z[\nu](\gamma^0)}_{\Omega_\s^t} \\
      &\leq&  (1+t) \bar r(t) \int_0^t \mathcal{W}_s(\rho,\nu) \d s
\]
with $\bar r(t)$ defined in \eqref{barr}.
For the second term we use the estimate \eqref{LipschZZ} of Lemma \ref{lem:LipschZZ}.
We thus arrive at
\[
   \mathcal{W}_t(\rho, \nu) \leq (1+t) \bar r(t) \int_0^t \mathcal{W}_s(\rho,\nu) \d s +  L_Z^t \mathcal{W}_0(\rho^0, \nu^0),
\]
where $L_Z^t$ is the constant provided by Lemma \ref{lem:LipschZ}.
We conclude by an application of the Gronwall lemma for $t\in [0,T]$, noting that both $(1+t)\bar r(t)$ and $L_Z^t$
are bounded on compact intervals.
\end{proof}

%%%%%%%%%%%%%%%%%%%%%%%%%%%%%
\subsection{Remark about Fokker-Planck description}\label{subsec:FP}

A natural question to ask is whether one can express the mean-field limit of the system
\eqref{eq:tau}--\eqref{eq:CS2} in terms of a Fokker-Planck equation
for some time dependent phase-space particle density $g_t \in \mathcal{P}(\R^d\times\R^d)$, $t\geq 0$.
An obvious candidate for such particle density is the push-forward measure $g_t := (X_t,V_t)\#\rho$,
where the maps $X(t)$ and $V(t)$ were defined in \eqref{XtVt},
and $\rho\in\POsT$ is a solution of the mean-field problem constructed in Theorem \ref{thm:MF}.
Then, taking any test function $\varphi=\varphi(x,v)$ and applying formally
the change-of-coordinates formula for the push-forward $g_t := (X_t,V_t)\#\rho$, we have
\[
   \tot{}{t} \iint_{\R^d\times\R^d} \varphi(x,v) \, \d g_t(x,v) &=& \tot{}{t} \int_{\Omega_\s^T} \varphi(X_t[\gamma], V_t[\gamma]) \, \d\rho(\gamma) \\
      &=&  \int_{\Omega_\s^T} \grad_x\varphi\cdot \dot X_t[\gamma] + \grad_v\varphi\cdot \dot V_t[\gamma] \,\d\rho(\gamma) \\
      &=&  \int_{\Omega_\s^T} \grad_x\varphi\cdot  V_t[\gamma] + \grad_v\varphi\cdot F_t[\rho](X_t[\gamma], V_t[\gamma]) \, \d\rho(\gamma),
\]
where we used the characteristic system \eqref{eq:MF1}--\eqref{eq:MF2} in the last equality.
Now, if we were able to express $F_t[\rho]$ in terms of $g$ as $G_t[g]$,
we could reverse the change-of-coordinates formula to obtain
\[
      \int_{\Omega_\s^T} \grad_x\varphi\cdot  V_t[\gamma]  + \grad_v\varphi\cdot F_t[\rho](X_t[\gamma], V_t[\gamma]) \, \d\rho(\gamma) 
      = \iint_{\R^d\times\R^d} \grad_x\varphi\cdot  v + \grad_v\varphi\cdot G_t[g](x,v) \, \d g_t(x,v).
\]
Then we could formally describe the mean-field limit by the Fokker-Planck equation
\(   \label{eq:FP}
   \partial_t g_t + v\cdot\grad_x g_t + \grad_v\cdot (G_t[g] g_t) = 0.
\)
Note that with \eqref{def:F} we have
\[
   F_t[\rho](x,v) &=&
      \int_{\Omega_\s^T} \psi\left( \left|\Gamma_{t,x}[\gamma] - x \right|\right) \left(\Pi_{t,x}[\gamma] - v \right)  \d\rho(\gamma) \\
      &=&
      \int_{\Omega_\s^T} \psi\left( \left|\gamma(t-\tau_{t,x}[\gamma])- x \right|\right) \left(\dot \gamma(t-\tau_{t,x}[\gamma]) - v \right)  \d\rho(\gamma) \\
      &=&
      \int_{\Omega_\s^T} \psi\left( \left|X_{t-\tau_{t,x}[\gamma]}[\gamma]- x \right|\right) \left(V_{t-\tau_{t,x}[\gamma]}[\gamma] - v \right)  \d\rho(\gamma),
\]
where $\tau_{t,x}[\gamma]:=\tau$ is the unique solution of \eqref{eq:tau_gamma}, i.e.,
\[
   \c\tau = | x - X_{t-\tau}[\gamma]|.
\]
However, the change-of-variables formula corresponding
to the push-forward $g_t = (X_t,V_t)\#\rho$ cannot be applied here since the evaluation times
in $X_{t-\tau_{t,x}[\gamma]}$ and $V_{t-\tau_{t,x}[\gamma]}$
depend on $\gamma$. Consequently, the operator $F_t[\rho]$ does not seem
to admit an equivalent form in terms of the push-forward measure $g_t = (X_t,V_t)\#\rho$.
In particular, one may be tempted to believe that the mean-field limit of \eqref{eq:tau}--\eqref{eq:CS2} 
should be given by \eqref{eq:FP} with the operator $G_t[f]$ taking the form
\[
   G_t[g](x,v) = \iint_{\R^d\times\R^d} \psi(|y-x|) (w-v) \, g(t-\c^{-1}|x-y|, y, w) \, \d y \,\d w.
\]
However, apart from the obvious difficulties with giving a meaning to the expression $g(t-\c^{-1}|x-y|, y, w) \, \d y \,\d w$
for measure-valued $g$, the above argument indicates that such an intuitive expectation is wrong.
We therefore conclude that the mean-field limit does not admit a description in terms of the (classical) Fokker-Planck equation \eqref{eq:FP},
and one has indeed to resort to the formulation with probability measures on the space of time-dependent trajectories,
as we did in this paper.

%%%%%%%%%%%%%%%%%%%%%%%%%%%%%%%%%%
\section*{Acknowledgment}
The author acknowledges the support of the KAUST baseline funds.
He also acknowledges the fruitful discussions with Oliver Tse that have taken place during his visit of TU Eindhoven,
and with Jan Vyb\'\i ral during his visit of Czech Technical University in Prague,
which helped to initiate and develop some ideas presented in this paper.

%%%%%%%%%%%%%%%%%%%%%%%%%%%%%%%%%%


\begin{thebibliography}{99}

\bibitem{Camazine}
S. Camazine, J. L. Deneubourg, N.R. Franks, J. Sneyd, G. Theraulaz and E. Bonabeau:
\emph{Self-Organization in Biological Systems}. Princeton University Press, Princeton, NJ, 2001.

\bibitem{CCR}
\newblock {J. Ca\~nizo, J. Carrillo and J. Rosado},
\newblock {A well-posedness theory in measures for some kinetic models of collective motion},
\newblock  \emph{Math. Mod. Meth. Appl. Sci.}, \textbf{21} (2011), 515--539.

\bibitem{Cartabia}
M. R. Cartabia: \emph{The Cucker-Smale model with time delay}.
{\tt arxiv.org/abs/2008.09530} (2020).

%\bibitem{Castellano}
%C. Castellano, S. Fortunato and V. Loreto: \emph{Statistical physics of social dynamics}.
%Rev. Mod. Phys., 81, (2009), 591--646.

\bibitem{ChoiH1}
Y.-P. Choi, J. Haskovec: \emph{Cucker-Smale model with normalized communication weights and time delay.}
Kinetic and Related Models 10 (2017), 1011-1033.
%{\tt doi:} 10.3934/krm.2017040.

\bibitem{ChoiH2}
Y.-P. Choi, J. Haskovec: \emph{Hydrodynamic Cucker-Smale model with normalized communication weights and time delay.}.
SIAM J. Math. Anal., Vol.~51, No.~3 (2019), 2660--2685.
%{\tt doi:} 10.1137/17M1139151.

\bibitem{Choi-Pignotti}
Y.-P. Choi and C. Pignotti: \emph{Emergent behavior of Cucker-Smale model with normalized weights and distributed time delays.}
Networks and Heterogeneous Media, Vol. 14 (2019), pp. 789--804.

\bibitem{CS1}
\newblock {F. Cucker and S. Smale},
\newblock {Emergent behaviour in flocks},
\newblock  \emph{IEEE T. on Automat. Contr.}, \textbf{52} (2007), 852--862.

\bibitem{CS2}
\newblock {F. Cucker and S. Smale},
\newblock {On the mathematics of emergence},
\newblock  \emph{Jap. J. Math.}, \textbf{2} (2007), 197--227.

\bibitem{EHS}
\newblock {R. Erban, J. Haskovec and Y. Sun},
\newblock {A Cucker-Smale model with noise and delay},
\newblock  \emph{SIAM J. Appl. Math.}, \textbf{76} (2016), 1535--1557.

%\bibitem{Halanay}
%A. Halanay, \emph{Differential Equations: Stability, Oscillations, Time Lags.}
%Academic Press, New York London, 1966.

\bibitem{Tadmor-Ha}
\newblock {S.-Y. Ha and E. Tadmor},
\newblock {From particle to kinetic and hydrodynamic descriptions of flocking},
\newblock  \emph{Kinetic and Related models}, \textbf{1} (2008), 315--335.

\bibitem{Hamman}
H. Hamman: \emph{Swarm Robotics: A Formal Approach}.
Springer, 2018.

\bibitem{H:SIADS}
{J. Haskovec},
\emph{A simple proof of asymptotic consensus in the Hegselmann-Krause
and Cucker-Smale models with normalization and delay.}
SIAM J. on Applied Dynamical Systems, 20:1 (2021), pp. 130--148.
%{\tt doi:} 10.1137/20m1341350.

\bibitem{Has:sdHK}
{J. Haskovec},
\emph{Asymptotic consensus in the Hegselmann-Krause model with finite speed of information propagation.}
Proc. Amer. Math. Soc. 149 (2021), 3425--3437.
%Preprint available at {\tt arxiv.org/abs/2005.05400}.

\bibitem{HasMar}
{J. Haskovec and I. Markou},
\emph{Asymptotic flocking in the Cucker-Smale model with reaction-type delays in the non-oscillatory regime.}
Kinetic and Related Models 13 (2020), 795--813.
%{\tt doi:} 10.3934/krm.2020027.

\bibitem{HasMar2}
J. Haskovec and I. Markou,
\emph{Exponential asymptotic flocking in the Cucker-Smale model with distributed reaction delays.}
Mathematical Biosciences and Engineering 17:5 (2020), pp. 5651--5671.
%{\tt doi:} 10.3934/mbe.2020304.

\bibitem{HK}
R. Hegselmann and U. Krause, \emph{Opinion dynamics and bounded confidence models, analysis, and simulation}.
J. Artif. Soc. Soc. Simul., 5, (2002), 1--24.

\bibitem{Jadbabaie}
A. Jadbabaie, J. Lin and A. S. Morse: \emph{Coordination of groups of mobile autonomous agents
using nearest neighbor rules}.
IEEE Trans. Automat. Control, 48, (2003), 988--1001.

\bibitem{Krugman}
P. Krugman: \emph{The Self Organizing Economy.} Blackwell Publishers, 1995.
%ISBN 1557866996

\bibitem{Liu-Wu}
\newblock {Y. Liu, and J. Wu},
\newblock {Flocking and asymptotic velocity of the Cucker-Smale model with processing delay},
\newblock  \emph{J. Math. Anal. Appl.}, \textbf{415} (2014), 53--61.

\bibitem{Naldi}
G. Naldi, L. Pareschi and G. Toscani (eds.):
\emph{Mathematical Modeling of Collective behaviour in Socio-Economic and Life Sciences},
Series: Modelling and Simulation in Science and Technology, Birkh\"auser, 2010.

%\bibitem{E2}
%S.-I. Niculescu: \emph{Delay Effects on Stability. A Robust Control Approach.}
%Springer-Verlag London, 2001.

\bibitem{Pignotti-Reche1}
C. Pignotti and I. Reche Vallejo: \emph{Asymptotic analysis of a Cucker-Smale system with leadership and distributed delay.}
In: Trends in Control Theory and Partial Differential Equations, Springer Indam Series, Vol. 32 (2019), pp. 233--253.

\bibitem{Pignotti-Reche2}
C. Pignotti and I. Reche Vallejo: \emph{Flocking estimates for the Cucker-Smale model with time lag and hierarchical leadership.}
Journal of Mathematical Analysis and Applications, Vol. 464 (2018), pp. 1313--1332.

\bibitem{Pignotti-Trelat}
C. Pignotti and E. Trelat:
\emph{Convergence to consensus of the general finite-dimensional
Cucker-Smale model with time-varying delays.}
Comm. Math. Sci. 16 (2018), 2053--2076.

\bibitem{Smith}
{H. Smith:}
\emph{An Introduction to Delay Differential
Equations with Applications to the Life Sciences.}
Springer New York Dordrecht Heidelberg London, 2011.

\bibitem{E6}
C. Somarakis and J. Baras:
\emph{Delay-independent convergence for linear consensus networks with applications to non-linear flocking systems}.
In Proceedings of the 12th IFAC Workshop on Time Delay Systems, pp. 159--164, Ann Arbor (2015).

\bibitem{E3B}
K. Szwaykowska, I.B. Schwartz, L.M. Romero, C.R. Heckman, D. Mox and M. Ani Hsieh:
\emph{Collective motion patterns of swarms with delay coupling: theory and experiment}.
Phys. Rev. E 93, 032307.

\bibitem{Valentini}
G. Valentini: \emph{Achieving Consensus in Robot Swarms: Design and Analysis of Strategies for the best-of-n Problem.}
Springer, Studies in Computational Intelligence, Vol. 706, 2017.

\bibitem{Vicsek}
T. Vicsek and A. Zafeiris: \emph{Collective motion}.
Phys. Rep., 517 (2012), 71--140.

\bibitem{old-new}
C. Villani: \emph{Optimal Transport. Old and New.}
Springer (2009).


\end{thebibliography}
\end{document}